\documentclass[a4paper,12pt]{article}
\usepackage[english]{babel}
\usepackage{tikz}
\usetikzlibrary{matrix,arrows,decorations.markings}
\usepackage{amsmath,amsfonts,amssymb,amsthm,url,textcomp}
\usepackage{csquotes}
\usepackage{a4wide}
\usepackage[numbers]{natbib}
\usepackage{fancyhdr}
\usepackage{anyfontsize}
\usepackage{hyperref}
\usepackage{hhline}
\usepackage{leftidx}

\allowdisplaybreaks

\newtheorem{theorem}{Theorem}\numberwithin{theorem}{section}
\newtheorem{definition}[theorem]{Definition}
\newtheorem{lemma}[theorem]{Lemma}

\newtheorem{conjecture}[theorem]{Conjecture}
\newtheorem{theoremm}{Theorem}\numberwithin{theoremm}{subsection}
\newtheorem{deffinition}[theoremm]{Definition}
\newtheorem{lemmma}[theoremm]{Lemma}

\newtheorem{nottation}[theoremm]{Notation}

\numberwithin{theoremmm}{subsubsection}

\theoremstyle{remark}
\newtheorem{remark}[theorem]{Remark}
\newtheorem{example}[theorem]{Example}

\newcommand{\Rad}{\operatorname{Rad}}
\newcommand{\Aut}{\operatorname{Aut}}
\newcommand{\Alt}{\mathcal{A}}

\newcommand{\Sym}{\mathcal{S}}

\newcommand{\A}{\operatorname{A}}

\newcommand{\C}{\operatorname{C}}

\newcommand{\Soc}{\operatorname{Soc}}
\newcommand{\Inn}{\operatorname{Inn}}
\newcommand{\id}{\operatorname{id}}

\newcommand{\e}{\mathrm{e}}

\newcommand{\GL}{\operatorname{GL}}

\newcommand{\G}{\mathcal{G}}

\newcommand{\D}{\operatorname{D}}

\newcommand{\Mat}{\operatorname{Mat}}

\renewcommand{\P}{\mathfrak{P}}

\newcommand{\B}{\operatorname{B}}

\renewcommand{\t}{\operatorname{t}}

\newcommand{\ind}{\operatorname{ind}}
\newcommand{\p}{\mathfrak{p}}
\newcommand{\conj}{\operatorname{conj}}
\newcommand{\IN}{\mathbb{N}}
\newcommand{\Hcal}{\mathcal{H}}
\newcommand{\Sp}{\operatorname{Sp}}
\newcommand{\perm}{\mathrm{perm}}
\newcommand{\res}{\operatorname{res}}
\newcommand{\aut}{\operatorname{aut}}
\newcommand{\IF}{\mathbb{F}}
\newcommand{\Stab}{\operatorname{Stab}}
\newcommand{\Span}{\operatorname{Span}}
\newcommand{\End}{\operatorname{End}}
\newcommand{\GU}{\operatorname{GU}}

\begin{document}

\title{Fibers of automorphic word maps and an application to composition factors}

\author{Alexander Bors\thanks{University of Salzburg, Mathematics Department, Hellbrunner Stra{\ss}e 34, 5020 Salzburg, Austria. \newline E-mail: \href{mailto:alexander.bors@sbg.ac.at}{alexander.bors@sbg.ac.at} \newline The author is supported by the Austrian Science Fund (FWF):
Project F5504-N26, which is a part of the Special Research Program \enquote{Quasi-Monte Carlo Methods: Theory and Applications}. \newline 2010 \emph{Mathematics Subject Classification}: Primary: 20D05, 20D06, 20D45. Secondary: 15A03, 15A04, 20B30. \newline \emph{Key words and phrases:} Finite groups, Word maps, Composition factors, Finite simple groups.}}

\date{\today}

\maketitle

\abstract{In this paper, we study the fibers of \enquote{automorphic word maps}, a certain generalization of word maps, on finite groups and on nonabelian finite simple groups in particular. As an application, we derive a structural restriction on finite groups $G$ where, for some fixed nonempty reduced word $w$ in $d$ variables and some fixed $\rho\in\left(0,1\right]$, the word map $w_G$ on $G$ has a fiber of size at least $\rho|G|^d$: No sufficiently large alternating group and no (classical) simple group of Lie type of sufficiently high rank can occur as a composition factor of such a group $G$.}

\section{Introduction}\label{sec1}

\subsection{Motivation and main result}\label{subsec1P1}

Word maps on groups have been studied intensely in recent years, resulting in substantial progress on interesting questions and a beautiful theory using tools from various areas such as representation theory and algebraic geometry; interested readers are referred to the survey article \cite{Sha13a}.

Recall that a (reduced) word $w$ in $d$ variables $X_1,\ldots,X_d$ is an element of the free group $F(X_1,\ldots,X_d)$. Each such word gives, for each group $G$, rise to a word map $w_G:G^d\rightarrow G$ induced by substitution. Studying the fibers of $w_G$ means studying the solution sets in $G^d$ to equations of the form $w=w(X_1,\ldots,X_d)=g$ for $g\in G$. By Larsen and Shalev's result \cite[Theorem 1.1]{LS12a}, for fixed $w$, the maximum number of solutions to such an equation in a nonabelian finite simple group $S$ is in $o(|S|^d)$ as $|S|\to\infty$. In particular, for each fixed number $\rho\in\left(0,1\right]$, for only finitely many nonabelian finite simple groups $S$, $w_S$ has a fiber of size at least $\rho|S|^d$.

Based on this, it is near-lying to ask what one can say more generally about the nonabelian composition factors of a finite group $G$ where the word map $w_G$ has a fiber of size at least $\rho|G|^d$. In order to be able to use \cite[Theorem 1.1]{LS12a} for this, it would be useful if one could somehow relate the maximum fiber size of $w_G$ with the maximum fiber sizes of the word maps associated with $w$ over the composition factors of $G$. For example, it would be nice to have an inequality of the form $\Pi_w(G)\leq\Pi_w(N)\cdot\Pi_w(G/N)$ for all finite groups $G$ and all normal subgroups $N$ of $G$, where $\Pi_w$ denotes the function that maps each finite group $G$ to the maximum fiber size of $w_G$. Unfortunately, this is not the case, even if we assume that $N$ is characteristic in $G$; consider, for example, $G=\D_{2o}$, the dihedral group of order $2o$, for some odd integer $o\geq 3$, $N$ the unique cyclic subgroup of index $2$ in $G$, and $w=X_1^2$.

In this paper, we will describe a way to circumvent these difficulties and provide some strong restrictions on possible composition factors of a finite group $G$ such that $\Pi_w(G)\geq\rho|G|^d$ in the form of Theorem \ref{wordMapTheo} below. First, we introduce some constants:

\begin{nottation}\label{mNot}
Let $w$ be a reduced word of length $l\geq1$ in $d$ distinct variables. We introduce the following constant, depending only on $w$:

\[
M=M(d,l):=1+2l(d+1)+(2l(d+1))^2+\cdots+(2l(d+1))^{2l+2}=\frac{(2l(d+1))^{2l+3}-1}{2l(d+1)-1}.
\]

Furthermore, we set $M':=M(l,l)$.
\end{nottation}

Our main result is the following (as usual, the \enquote{untwisted Lie rank} of a Lie type group is the Lie rank of the corresponding untwisted group):

\begin{theoremm}\label{wordMapTheo}
Let $w$ be a reduced word of length $l\geq1$ in $d$ distinct variables. Then for all $\rho\in\left(0,1\right]$ and all finite groups $G$ such that the word map $w_G$ has a fiber of size at least $\rho|G|^d$, the following hold:

\begin{enumerate}
\item No alternating group of order larger than

\[
\max\{\lceil 256l^{16}\e^{16M'l-2}\rceil!,\rho^{-16M'}\}
\]

is a composition factor of $G$.
\item No (classical) simple group of Lie type of untwisted Lie rank larger than

\[
\max\{72(l+1)^2l^2,\sqrt{72(l+1)^2l^2\log_2(\rho^{-1})}\}
\]

is a composition factor of $G$.
\end{enumerate}
\end{theoremm}

In other words, the list of potential composition factors for such a group $G$ consists of finitely many alternating groups, the sporadic groups, and all simple Lie type groups of bounded rank.

\subsection{Main ideas and overview of the paper}\label{subsec1P2}

The main idea for proving Theorem \ref{wordMapTheo} is to make up for the above mentioned \enquote{flaw} of the function $\Pi_w$ by replacing it by an evaluation-wise larger function $\P_w$ which satisfies the inequality $\P_w(G)\leq\P_w(N)\cdot\P_w(G/N)$ at least when $N$ is characteristic in $G$ and study $\P_w$ instead. To this end, we generalize the notion of a word map in a certain way.

Let us first fix some notation. For a fixed reduced word $w$ in the $d$ variables $X_1,\ldots,X_d$ and of length $l$, write $w=x_1^{\epsilon_1}\cdots x_l^{\epsilon_l}$, where $x_1,\ldots,x_l\in\{X_1,\ldots,X_d\}$ and $\epsilon_i=\pm 1$. Denote by $\iota$ the unique function $\{1,\ldots,l\}\rightarrow\{1,\ldots,d\}$ such that for $i=1,\ldots,l$, $x_i=X_{\iota(i)}$. Thus for each group $G$, the word map $w_G$ is just the map $G^d\rightarrow G$ sending $(g_1,\ldots,g_d)\mapsto g_{\iota(1)}^{\epsilon_1}\cdots g_{\iota(l)}^{\epsilon_l}$.

\begin{deffinition}\label{automorphicDef}
We introduce the following terminology and notation:

\begin{enumerate}
\item With notation as above, let $G$ be a group, and let $\alpha_1,\ldots,\alpha_l$ be automorphisms of $G$. The \emph{automorphic word map $w_G^{(\alpha_1,\ldots,\alpha_l)}$} is the map $G^d\rightarrow G$ sending $(g_1,\ldots,g_d)\mapsto \alpha_1(g_{\iota(1)})^{\epsilon_1}\cdots\alpha_l(g_{\iota(l)})^{\epsilon_l}$.
\item By $\P_w$, we denote the function that maps each finite group $G$ to the maximum size of a fiber of one of the automorphic word maps $w_G^{(\alpha_1,\ldots,\alpha_l)}$, $\alpha_1,\ldots,\alpha_l$ automorphisms of $G$.
\end{enumerate}
\end{deffinition}

Hence $w_G^{(\alpha_1,\ldots,\alpha_l)}$ is like $w_G$, except that in the $i$-th factor of the $l$ factor product as which the evaluation $w_G(g_1,\ldots,g_d)$ is defined, we additionally apply $\alpha_i$, one of $l$ automorphisms fixed beforehand. In particular, $w_G^{(\id,\ldots,\id)}=w_G$.

The approach of studying fibers of automorphic word maps will actually allow us to prove the following stronger form of Theorem \ref{wordMapTheo}:

\begin{theoremm}\label{mainTheo}
Let $w$ be a reduced word of length $l\geq1$ in $d$ distinct variables and $M=M(w)$ as in Notation \ref{mNot}. Then for all $\rho\in\left(0,1\right]$ and all finite groups $G$ with $\P_w(G)\geq\rho|G|^d$, the following hold:

\begin{enumerate}
\item No alternating group of order larger than

\[
\max\{\lceil 256l^{16}\e^{16M'l-2}\rceil!,\rho^{-16M'}\}
\]

is a composition factor of $G$.
\item No (classical) simple group of Lie type of untwisted Lie rank larger than

\[
\max\{72(l+1)^2l^2,\sqrt{72(l+1)^2l^2\log_2(\rho^{-1})}\}
\]

is a composition factor of $G$.
\end{enumerate}
\end{theoremm}

We now give an overview of the rest of this paper:

\begin{enumerate}
\item In Section \ref{sec2}, we prove our main lemma, Lemma \ref{mainLem}, which includes the inequality $\P_w(G)\leq\P_w(N)\cdot\P_w(G/N)$ for characteristic subgroups $N$ of $G$. It also includes the observation that the element of $G$ having the largest fiber size under any automorphic word map on $G$ is the identity element of $G$.
\item Having gained a basic understanding of automorphic word maps in Section \ref{sec2}, the next goal is to extend, as far as necessary, Larsen and Shalev's result \cite[Theorem 1.1]{LS12a} on fibers of word maps on nonabelian finite simple groups mentioned above to fibers of automorphic word maps. This will be done in Section \ref{sec3}, see Theorem \ref{simpleGroupsTheo}.
\item Section \ref{sec4} consists of the proof of Theorem \ref{mainTheo} based on the results developed so far.
\item Finally, in Section \ref{sec5}, we give some concluding remarks concerning further extensions of Larsen and Shalev's techniques to automorphic word maps and an interesting consequence thereof.
\end{enumerate}

\subsection{Notation}\label{subsec1P3}

We denote by $\IN$ the set of natural numbers (including $0$) and by $\IN^+$ the set of positive integers. Euler's constant is denoted by $\e$, which is to be distinguished from the variable $e$. The image and preimage of a set $M$ under a function $f$ are denoted by $f[M]$ and $f^{-1}[M]$ respectively. When $f_i:X_i\rightarrow Y_i$ for $i=1,\ldots,n$, then we denote by $f_1\times\cdots\times f_n$ the \emph{product of the maps $f_i$}, i.e., the map $\prod_{i=1}^n{X_i}\rightarrow\prod_{i=1}^n{Y_i},(x_1,\ldots,x_n)\mapsto(f_1(x_1),\ldots,f_n(x_n))$. The $n$-fold product of a map $f$ with itself is denoted by $f^{(n)}$. These last two notations will be used in the proofs of Lemma \ref{variationLem} and of the implication \enquote{Conjecture \ref{larShaEffConj}$\Rightarrow$ Conjecture \ref{radicalConj}} in Section \ref{sec5}.

For a group $G$ and an element $g\in G$, we denote by $\conj(g)$ the conjugation by $g$ on $G$, i.e., the inner automorphism of $G$ of the form $x\mapsto gxg^{-1}$. The automorphism group of $G$ is denoted by $\Aut(G)$, and the inner automorphism group of $G$ by $\Inn(G)$. For a finite set $X$, we denote by $\Sym_X$ the symmetric group on $X$; for a positive integer $n$, $\Sym_n$ and $\Alt_n$ denote the symmetric and alternating group on $\{1,\ldots,n\}$ respectively.

For a prime power $q$, the finite field with $q$ elements is denoted by $\IF_q$. For $n\in\IN^+$ and a prime power $q$, $\Mat_n(q)$ denotes the ring of $(n\times n)$-matrices over $\IF_q$. For a vector space $\Delta$ over some field $F$, we denote by $\End_F(\Delta)$ the endomorphism ring of $\Delta$, i.e., the ring of $F$-linear maps $\Delta\rightarrow\Delta$.

At some points in our arguments, we will not consider all possible automorphic word maps $w_G^{(\alpha_1,\ldots,\alpha_l)}$ over some finite group $G$, but only those where the $\alpha_i$ are from a certain subset of $\Aut(G)$. Also, we sometimes want to talk about the maximum fiber size of a particular element of $G$ under an automorphic word map or about the proportion of a fiber of an (automorphic) word map associated with $w$ within the entire argument set $G^d$, rather than the actual size of the fiber. We therefore introduce the following notation that supplements the notation already introduced:

\begin{nottation}\label{pANot}
Let $G$ be a finite group, $w$ a reduced word of length $l$ in $d$ distinct variables, $A\subseteq\Aut(G)$.

\begin{enumerate}
\item We set $\pi_w(G):=\Pi_w(G)/|G|^d$ and $\p_w(G):=\P_w(G)/|G|^d$. Note that always $\pi_w(G),\p_w(G)\in\left(0,1\right]$.
\item We denote by $\P_w^{(A)}(G,g)$ the maximum size of the fiber of $g$ under an automorphic word map of the form $w_G^{(\alpha_1,\ldots,\alpha_l)}$, where $\alpha_i\in A$ for $i=1,\ldots,l$, and we set $\P_w^{(A)}(G):=\max_{g\in G}{\P_w^{(A)}(G,g)}$ (so that $\P_w^{(\Aut(G))}(G)=\P_w(G)$).
\item Moreover, we set $\p_w^{(A)}(G,g):=\P_w^{(A)}(G,g)/|G|^d$ and $\p_w^{(A)}(G):=\P_w^{(A)}(G)/|G|^d$.
\end{enumerate}
\end{nottation}

\section{Basic results on automorphic word maps}\label{sec2}

In this section, we prove the following lemma containing some basic bounds on fiber sizes of automorphic word maps:

\begin{lemma}\label{mainLem}
Let $w$ be a reduced word, $G$ a finite group, $A$ a subgroup of $\Aut(G)$ containing $\Inn(G)$. Furthermore, let $N$ be a characteristic subgroup of $G$, and denote by

\begin{itemize}
\item $\ind(A)$ the subgroup of $\Aut(G/N)$ consisting of all automorphisms of $G/N$ induced by some automorphism from $A$,
\item $\res(A)$ the subgroup of $\Aut(N)$ consisting of all restrictions of automorphisms from $A$ to $N$,
\item $\pi:G\rightarrow G/N$ the canonical projection.
\end{itemize}

Then the following hold:

\begin{enumerate}
\item For all $g\in G$, $\P_w^{(A)}(G,g)\leq\P_w^{(\ind(A))}(G/N,\pi(g))\cdot\P_w^{(\res(A))}(N,1)$, or in terms of proportions, $\p_w^{(A)}(G,g)\leq\p_w^{(\ind(A))}(G/N,\pi(g))\cdot\p_w^{(\res(A))}(N,1)$.
\item For all $g\in G$, $\P_w^{(A)}(G,g)\leq\P_w^{(A)}(G,1)$, or in terms of proportions, $\p_w^{(A)}(G,g)\leq\p_w^{(A)}(G,1)$. Hence $\P_w^{(A)}(G,1)=\P_w^{(A)}(G)$.
\item $\P_w^{(A)}(G)\leq\P_w^{(\ind(A))}(G/N)\cdot\P_w^{(\res(A))}(N)$, or in terms of proportions, $\p_w^{(A)}(G)\leq\p_w^{(\res(A))}(G/N)\cdot\p_w^{(\ind(A))}(N)$.
\end{enumerate}
\end{lemma}

\begin{proof}
For (1): As before, we write $w=x_1^{\epsilon_1}\cdots x_l^{\epsilon_l}$ with $\epsilon_i\in\{\pm1\}$, $x_1,\ldots,x_l\in\{X_1,\ldots,X_d\}$ and $\iota:\{1,\ldots,l\}\rightarrow\{1,\ldots,d\}$ such that $x_i=X_{\iota(i)}$. Furthermore, fix an $l$-tuple $(\alpha_1,\ldots,\alpha_l)$ of elements of $A$ such that the size of the fiber $\Phi$ of $g$ under $w_G^{(\alpha_1,\ldots,\alpha_l)}$ equals $\P_w^{(A)}(G,g)$. For $i=1,\ldots,l$, denote by $\tilde{\alpha_i}$ the automorphism of $G/N$ induced by $\alpha_i$.

We will establish the inequality by a coset-wise counting argument. More precisely, we will show the following two assertions, which together imply the inequality:

\begin{enumerate}
\item The number of cosets of $N^d$ in $G^d$ having nonempty intersection with $\Phi$ is at most $\P_w^{(\ind(A))}(G/N,\pi(g))$.
\item $\Phi$ intersects each coset of $N^d$ in $G^d$ in at most $\P_w^{(\res(A))}(N,1)$ many elements.
\end{enumerate}

For the first assertion, let $(g_1,\ldots,g_d)\in\Phi$. In other words,

\begin{equation}\label{eq1}
\alpha_1(g_{\ind(1)})^{\epsilon_1}\cdots\alpha_l(g_{\ind(l)})^{\epsilon_l}=g.
\end{equation}

Applying $\pi$ to both sides of Formula (\ref{eq1}) yields

\[\tilde{\alpha_1}(\pi(g_{\ind(1)}))\cdots\tilde{\alpha_l}(\pi(g_{\ind(l)}))=\pi(g),\]

and thus that $(\pi(g_1),\ldots,\pi(g_d))$ lies in the fiber of $\pi(g)$ under $w_{G/N}^{(\tilde{\alpha_1},\ldots,\tilde{\alpha_l})}$. The assertion follows immediately from this.

For the second assertion, fix a coset $C$ of $N^d$ in $G^d$, say $C=N^d(g_1,\ldots,g_d)$. We want to show that $|C\cap\Phi|\leq\P_w^{(\res(A))}(N,1)$. Of course, we may assume that $C\cap\Phi$ is nonempty, and w.l.o.g.~even that the coset representative $(g_1,\ldots,g_d)$ which we fixed lies in $\Phi$. Hence Formula (\ref{eq1}) holds. We now characterize those $(n_1,\ldots,n_d)\in N^d$ such that $w_G^{(\alpha_1,\ldots,\alpha_l)}(n_1g_1,\ldots,n_dg_d)=g$ as well, i.e., such that

\begin{equation}\label{eq2}
g=\alpha_1(n_{\iota(1)}g_{\iota(1)})^{\epsilon_1}\cdots\alpha_l(n_{\iota(l)}g_{\iota(l)})^{\epsilon_l}=t_1\cdots t_l,
\end{equation}

where

\[
t_i=\begin{cases}\alpha_i(n_{\iota(i)})^{\epsilon_i}\alpha_i(g_{\iota(i)})^{\epsilon_i}, & \text{if }\epsilon_i=+1, \\ \alpha_i(g_{\iota(i)})^{\epsilon_i}\alpha_i(n_{\iota(i)})^{\epsilon_i}, & \text{if }\epsilon_i=-1.\end{cases}
\]

Note that under the assumed Formula (\ref{eq1}), Formula (\ref{eq2}) is equivalent to the following:

\begin{equation}\label{eq3}
1=g\cdot g^{-1}=t_1t_2\cdots t_l\cdot\alpha_l(g_{\iota(l)})^{-\epsilon_l}\cdots\alpha_2(g_{\iota(2)})^{-\epsilon_2}\alpha_1(g_{\iota(1)})^{-\epsilon_1}.
\end{equation}

We now transform the product expression on the RHS of Formula (\ref{eq3}) without changing its value as follows: The product has a unique subproduct of the form $\alpha_l(g_{\iota(l)})^{\epsilon_l}w_l\alpha_l(g_{\iota(l)})^{-\epsilon_l}$ (where $w_l$ is either empty or equal to $\alpha_l(n_{\iota(l)})^{\epsilon_l}$, depending on whether $\epsilon_l=1$ or $\epsilon_l=-1$). Replace this subproduct by the expression $\conj(\alpha_l(g_{\iota(l)})^{\epsilon_l})(w_1)$. The resulting product expression has a unique subproduct of the form $\alpha_{l-1}(g_{\iota(l-1)})^{\epsilon_{l-1}}w_{l-1}\alpha_{l-1}(g_{\iota(l-1)})^{-\epsilon_{l-1}}$. Replace this subproduct by $\conj(\alpha_{l-1}(g_{\iota(l-1)})^{\epsilon_{l-1}})(w_{l-1})$ and distribute the application of the automorphism $\conj(\alpha_{l-1}(g_{\iota(l-1)})^{\epsilon_{l-1}})$ onto the single factors of the expression $w_{l-1}$. Continuing in this fashion, we eventually receive an expression of the form $w_N^{(\beta_1,\ldots,\beta_l)}(n_1,\ldots,n_d)$, where each $\beta_i$ is the restriction to $N$ of an element of $A$, namely of the composition of some inner automorphism of $G$ with $\alpha_i$; see also Example \ref{mainLemEx} for an illustration.

For (2): This follows by setting $N:=G$ in point (1) of this lemma.

For (3): By points (1) and (2), we have

\begin{align*}
\P_w^{(A)}(G) &=\P_w^{(A)}(G,1)\leq\P_w^{(\ind(A))}(G/N,\pi(1))\cdot\P_w^{(\res(A))}(N,1) \\
&=\P_w^{(\ind(A))}(G/N)\cdot\P_w^{(\res(A))}(N),
\end{align*}

using that $\ind(A)$ resp. $\res(A)$ contains $\Inn(G/N)$ resp. $\Inn(N)$.
\end{proof}

\begin{example}\label{mainLemEx}
We illustrate the transformations of expressions described in the proof of Lemma \ref{mainLem}(1) by the following example: Say $w=[X_1,X_2]=X_1X_2X_1^{-1}X_2^{-1}$. Then Formula (\ref{eq1}) becomes

\[
\alpha_1(g_1)\alpha_2(g_2)\alpha_3(g_1)^{-1}\alpha_4(g_2)^{-1}=g,
\]

and Formula (\ref{eq2}) becomes

\begin{align*}
g &= \alpha_1(n_1g_1)\alpha_2(n_2g_2)\alpha_3(n_1g_1)^{-1}\alpha_4(n_2g_2)^{-1} \\
&= \alpha_1(n_1)\alpha_1(g_1)\alpha_2(n_2)\alpha_2(g_2)\alpha_3(g_1)^{-1}\alpha_3(n_1)^{-1}\alpha_4(g_2)^{-1}\alpha_4(n_2)^{-1}.
\end{align*}

Together, this yields

\begin{align*}
&1 = g\cdot g^{-1} \\
&= \alpha_1(n_1)\alpha_1(g_1)\alpha_2(n_2)\alpha_2(g_2)\alpha_3(g_1)^{-1}\alpha_3(n_1)^{-1}\cdot\alpha_4(g_2)^{-1}\alpha_4(n_2)^{-1}\alpha_4(g_2) \\ &\cdot\alpha_3(g_1)\alpha_2(g_2)^{-1}\alpha_1(g_1)^{-1} \\
&= \alpha_1(n_1)\alpha_1(g_1)\alpha_2(n_2)\alpha_2(g_2)\cdot\alpha_3(g_1)^{-1}\alpha_3(n_1)^{-1}(\conj(\alpha_4(g_2)^{-1})\circ\alpha_4)(n_2)^{-1}\alpha_3(g_1) \\ &\cdot\alpha_2(g_2)^{-1}\alpha_1(g_1)^{-1} \\
&= \alpha_1(n_1)\alpha_1(g_1)\alpha_2(n_2)\alpha_2(g_2)\cdot\conj(\alpha_3(g_1)^{-1})(\alpha_3(n_1)^{-1}(\conj(\alpha_4(g_2)^{-1})\circ\alpha_4)(n_2)^{-1}) \\ &\cdot\alpha_2(g_2)^{-1}\alpha_1(g_1)^{-1} \\
&= \alpha_1(n_1)\alpha_1(g_1)\alpha_2(n_2)\alpha_2(g_2)\cdot(\conj(\alpha_3(g_1)^{-1})\circ\alpha_3)(n_1)^{-1} \\ &\cdot(\conj(\alpha_3(g_1)^{-1}\alpha_4(g_2)^{-1})\circ\alpha_4)(n_2)^{-1}\cdot\alpha_2(g_2)^{-1}\alpha_1(g_1)^{-1} \\
&= \alpha_1(n_1)\alpha_1(g_1)\alpha_2(n_2)\cdot(\conj(\alpha_2(g_2)\alpha_3(g_1)^{-1})\circ\alpha_3)(n_1)^{-1} \\ &\cdot(\conj(\alpha_2(g_2)\alpha_3(g_1)^{-1}\alpha_4(g_2)^{-1})\circ\alpha_4)(n_2)^{-1}\cdot\alpha_1(g_1)^{-1} \\
&= \alpha_1(n_1)\cdot(\conj(\alpha_1(g_1))\circ\alpha_2)(n_2)\cdot(\conj(\alpha_1(g_1)\alpha_2(g_2)\alpha_3(g_1)^{-1})\circ\alpha_3)(n_1)^{-1} \\ &\cdot(\conj(\alpha_1(g_1)\alpha_2(g_2)\alpha_3(g_1)^{-1}\alpha_4(g_2)^{-1})\circ\alpha_4)(n_2)^{-1} \\
&= w_N^{(\alpha_1,\conj(\alpha_1(g_1))\circ\alpha_2,\conj(\alpha_1(g_1)\alpha_2(g_2)\alpha_3(g_1)^{-1})\circ\alpha_3,\conj(\alpha_1(g_1)\alpha_2(g_2)\alpha_3(g_1)^{-1}\alpha_4(g_2)^{-1})\circ\alpha_4)}(n_1,n_2),
\end{align*}

where in the last expression, all automorphisms are understood to be restricted to $N$.
\end{example}

\section{On fibers of automorphic word maps on nonabelian finite simple groups}\label{sec3}

\subsection{Larsen and Shalev's result and the main result of this section}\label{subsec3P1}

The following theorem is an equivalent reformulation of \cite[Theorem 1.1]{LS12a}:

\begin{theoremm}\label{larShaTheo}
For each nonempty and reduced word $w$ in $d$ distinct variables, there exist constants $N(w),\eta(w)>0$ such that for all nonabelian finite simple groups $S$ with $|S|\geq N(w)$, the inequality $\Pi_w(S)\leq |S|^{d-\eta(w)}$ holds.\qed
\end{theoremm}

The proof of this theorem in \cite{LS12a} is split into three parts (note that the sporadic groups can be ignored here, as the fiber size bound only needs to be shown for large enough $S$):

\begin{itemize}
\item First, the bound is established for large enough alternating groups by means of a certain combinatorial construction.
\item Next, the bound is established for all simple Lie type groups of sufficiently high rank, where the lower bound on the rank is so large that only classical groups need to be considered in this case. As Larsen and Shalev say themselves, the argument is conceptually similar to the one for alternating groups.
\item Finally, the simple Lie type groups of bounded rank are treated by means of an argument using results of algebraic geometry.
\end{itemize}

It turns out that Larsen and Shalev's arguments in the first two cases can be modified to prove the following, which is the main result of this section:

\begin{theoremm}\label{simpleGroupsTheo}
Let $w$ be a reduced word of length $l\geq 1$ in $d$ distinct variables, and $M=M(w)$ as in Notation \ref{mNot}. Then the following hold:

\begin{enumerate}
\item For all $n\in\IN^+$ with $n\geq 256l^{16}\e^{16Md-2}$, we have $\P_w(\Alt_n)\leq |\Alt_n|^{d-1/(16M)}$.
\item For all simple Lie type groups $S$ of untwisted Lie rank at least $72(d+1)^2l^2$, we have $\P_w(S)\leq |S|^{d-1/(72(d+1)l^2)}$.
\end{enumerate}
\end{theoremm}

Whether such bounds can also be established for the simple Lie type groups of \enquote{small} rank is open; see Section \ref{sec5} for some more remarks on this.

\subsection{Reduction of Theorem \ref{simpleGroupsTheo} to Theorem \ref{alternativeTheo}}\label{subsec3P2}

Similarly to \cite{LS12a}, the main part of the argument for Theorem \ref{simpleGroupsTheo} will not provide upper bounds on $\P_w(S)$ for the simple groups $S$ in question directly, but on $\P_w^{(A)}(G)$, where $G$ is a finite group \enquote{closely related with $S$} and $A$ a certain subgroup of $\Aut(G)$. That we can do this without loss of generality is justified by the following lemma, a modification of \cite[Lemma 2.1]{LS12a}, which served the same purpose:

\begin{lemmma}\label{similarGroupsLem}
Let $w$ be a nonempty reduced word in $d$ distinct variables, $\G$ and $\Hcal$ infinite classes of finite groups, $N,\eta>0$. Assume that for each $H\in\Hcal$ with $|H|\geq N$, there is associated a subgroup $\A(H)\leq\Aut(H)$ such that $\P_w^{(\A(H))}(H)\leq |H|^{d-\eta}$ (and note that this implies $\eta\leq d$). Set $\epsilon:=\eta/(2(1+d-\eta))>0$. Finally, assume that there exists $C>0$ such that for all $G\in\G$ with $|G|\geq C$, the following exist:

\begin{itemize}
\item an $H\in\Hcal$ such that $|H|\leq |G|^{1+\epsilon}$,
\item characteristic subgroups $K$ and $H'$ of $H$ with $K\leq H'$ such that $G\cong H'/K$ (we say that $G$ is a \emph{characteristic section of $H$}) and such that every automorphism of $G$ can be induced by the restriction to $H'$ of a suitable automorphism of $H$ from $\A(H)$.
\end{itemize}

Then the following holds: For all $G\in\G$ with $|G|\geq\max\{N,C\}$, $\P_w(G)\leq|G|^{d-\eta/2}$.
\end{lemmma}

\begin{proof}
Let $G\in\G$ with $|G|\geq\max\{N,C\}$. Fix $H\in\Hcal$ with $|H|\leq|G|^{1+\epsilon}$ containing characteristic subgroups $K$ and $H'$ as described in the assumptions. We assume w.l.o.g.~that $G=H'/K$ (not just isomorphic). Note that we have in particular that $|H|\geq|G|\geq N$, so that $\P_w^{(\A(H))}(H)\leq|H|^{d-\eta}$ by assumption.

We want to bound the fiber sizes of automorphic word maps over $G$. To this end, fix $\tilde{\alpha_1},\ldots,\tilde{\alpha_l}\in\Aut(G)$ and $g\in G$ such that the fiber size of $w_G^{(\tilde{\alpha_1},\ldots,\tilde{\alpha_l})}$ equals $\P_w(G)$ (note that since we are not assuming that $\A(H)$ contains $\Inn(H)$, we also cannot assume w.l.o.g.~here that $g=1$, as would otherwise follow from Lemma \ref{mainLem}(2)). Fix $h\in H'$ projecting onto $g\in G=H'/K$, and fix $\alpha_1,\ldots,\alpha_l\in\A(H)\leq\Aut(H)$ such that for $i=1,\ldots,l$, $(\alpha_i)_{\mid H'}$ induces $\tilde{\alpha_i}$ on $G$.

Since each fiber of $w_H^{(\alpha_1,\ldots,\alpha_l)}$, and thus in particular each fiber of $w_{H'}^{((\alpha_1)_{\mid H'},\ldots,(\alpha_l)_{\mid H'})}$, has size at most $|H|^{d-\eta}$, and since the fiber of $g$ under $w_G^{(\tilde{\alpha_1},\ldots,\tilde{\alpha_l})}$ can be expressed as the image under the canonical projection $H'\rightarrow G$ of a disjoint union of at most $|G|^{\epsilon}$ many fibers of $w_{H'}^{((\alpha_1)_{\mid H'},\ldots,(\alpha_l)_{\mid H'})}$, we get that

\[
\P_w(G)=|(w_G^{(\tilde{\alpha_1},\ldots,\tilde{\alpha_l})})^{-1}[\{g\}]|\leq |G|^{\epsilon}\cdot|H|^{d-\eta}\leq |G|^{\epsilon+(1+\epsilon)(d-\eta)}=|G|^{d-\eta/2},
\]

where the last equality is by definition of $\epsilon$.
\end{proof}

In accordance with our announcement before Lemma \ref{similarGroupsLem}, below, we will prove Lemmata \ref{altSymLem} and \ref{omegaILem}, which allow us to reduce the proof of Theorem \ref{simpleGroupsTheo} to the proof of a theorem concerning fibers of automorphic word maps in slightly different classes of groups, Theorem \ref{alternativeTheo}. For example, for the alternating groups, these \enquote{closely related} groups will be just the symmetric groups. To make the formulations of the lemmata shorter, let us first introduce the following terminology:

\begin{deffinition}\label{niceDef}
Let $w$ be a nonempty reduced word in $d$ distinct variables, and let $N,\eta>0$. Furthermore, let $\G$ be a class of finite groups and $\A$ a function that maps each $G\in\G$ to a subgroup $\A(G)\leq\Aut(G)$. We say that $\G$ is \emph{$(\A,N,\eta)$-nice for $w$}, or that \emph{$(\A,N,\eta)$ is a niceness tuple of $\G$ for $w$}, if and only if for all $G\in\G$ with $|G|\geq N$, we have $\P_w^{(\A(G))}(G)\leq |G|^{d-\eta}$.
\end{deffinition}

The following lemma allows us to reduce Theorem \ref{simpleGroupsTheo}(1) to the study of automorphic word map fibers in \emph{symmetric} groups:

\begin{lemmma}\label{altSymLem}
Let $w$ be a nonempty reduced word in $d$ distinct variables. Assume that for some $N\geq |\Alt_7|=2520$ and some $\eta>0$, the class of finite symmetric groups is $(\Aut,N,\eta)$-nice for $w$. Then the class of finite alternating groups is $(\Aut,\max\{N,2^{2(1+d-\eta)/\eta+1}\},\eta/2)$-nice for $w$.
\end{lemmma}

\begin{proof}
Set $\epsilon:=\eta/(2(1+d-\eta))$. We want to apply Lemma \ref{similarGroupsLem} with $\Hcal$ the class of finite symmetric groups and $\G$ the class of finite alternating groups. Let us first find $C_0>0$ such that for all $\Alt_n$ with $|\Alt_n|\geq C_0$, $|\Sym_n|\leq|\Alt_n|^{1+\epsilon}$. Taking logarithms, the inequality turns into $\log{n!}\leq(1+\epsilon)(\log{n!}-\log{2})$, which is equivalent to $\log{n!}\geq\log{2}\cdot(1+\epsilon)/\epsilon=\log{2}(1+1/\epsilon)=\log{2}(1+2(1+d-\eta)/\eta)$. Hence we want that $n!\geq 2^{1+2(1+d-\eta)/\eta)}$, which is satisfied if $|\Alt_n|\geq 2^{1+2(1+d-\eta)/\eta)}$, and so $C_0:=2^{1+2(1+d-\eta)/\eta)}$ does the job. Since by Lemma \ref{similarGroupsLem}, we also need to ensure that every automorphism of $\Alt_n$ is induced by an automorphism of $\Sym_n$, we need to set $C:=\max\{2520,C_0\}$. With this choice for $C$, an application of Lemma \ref{similarGroupsLem} yields that a possible niceness tuple of the class of finite alternating groups for $w$ is $(\Aut,\max\{N,C\},\eta/2)=(\Aut,\max\{N,2520,C_0\},\eta/2)=(\Aut,\max\{N,2^{1+2(1+d-\eta)/\eta}\},\eta/2)$, as required.
\end{proof}

As for the classical groups of Lie type $X(q)$ with which we are concerned in Theorem \ref{simpleGroupsTheo}(2), the groups \enquote{closely related} with them which we will study are, just as in \cite[beginning of Section 3]{LS12a}, the isometry groups of trivial, perfect symmetric, perfect anti-symmetric or perfect Hermitian pairings (depending on the case) of a vector space over either the field $\IF_q$ or (only in the Hermitian case) its degree $2$ extension $\IF_{q^2}$. Set $E:=\IF_q$, and moreover, set $F:=E$ except in the Hermitian case, where $F:=\IF_{q^2}$.

In the notation of \S2.1 in Kleidman and Liebeck's book \cite{KL90a}, the classical simple Lie type group $X(q)$ is $\overline{\Omega}$ and is the projective version of a subgroup $\Omega$ of the associated isometry group, which is denoted by $I$. $I$, in turn, is contained as a normal subgroup in some group $A$ which (by its conjugation action on $I$) may be viewed as a subgroup of $\Aut(I)$ and is just the group $\Gamma$ of collineations over $I$ except when $I=\GL_n(q)$ is the isometry group of a trivial form, in which case $A$ is the subgroup of $\Aut(I)$ generated by $\Gamma$ and the inverse-transpose automorphism of $I$. For later purposes, we set $\A(I):=A$.

It follows from \cite[Theorem 2.1.4]{KL90a} that if the untwisted Lie rank of $\overline{\Omega}$ is at least $5$ (this is just to exclude the groups $\C_2(2^f)=\Sp_4(2^f)$ and $\D_4(q)=\Omega_8^+(q)$), then every automorphism of $\overline{\Omega}$ is induced by the restriction to $\Omega$ of an automorphism of $I$ from $A=\A(I)$, as required in Lemma \ref{similarGroupsLem}. Furthermore, as Larsen and Shalev observe in \cite[beginning of Section 3]{LS12a}, we always have $|I|\leq (q-1)(r+1)|\overline{\Omega}|$, where $r$ is the untwisted Lie rank of $\overline{\Omega}$ and $q=|F|$. They also observe that for every $\epsilon>0$, $|\overline{\Omega}|^{\epsilon}\geq(r+1)(q-1)$ if $r$ is sufficiently large. This \enquote{sufficiently large} can be made explicit:

\begin{lemmma}\label{explicitLem}
For all $\epsilon>0$, the following holds: With notation as above, if $r\geq\epsilon^{-1}$, then $|\overline{\Omega}|^{\epsilon}\geq(r+1)(q-1)$, and so $|I|\leq|\overline{\Omega}|^{1+\epsilon}$.
\end{lemmma}

\begin{proof}
This follows from $|\overline{\Omega}|^{1/r}\geq (r+1)(q-1)$, which can be easily verified in each of the six cases $\overline{\Omega}=\A_r(q),\B_r(q),\C_r(q),\D_r(q),\leftidx{^{2}}\A_r(q),\leftidx{^{2}}\D_r(q)$ using the known formula for $|\overline{\Omega}|$.
\end{proof}

We can now show the following:

\begin{lemmma}\label{omegaILem}
Let $N,\eta>0$ and $\G$ a class of finite groups consisting only of the isometry groups $I$ associated with the members of a subclass $\Hcal$ of the class of finite classical simple Lie type groups $\overline{\Omega}$ of untwisted Lie rank at least $\max\{5,2(1+d-\eta)/\eta\}$. Assume that $\G$ is $(\A,N,\eta)$-nice for $w$. Then $\Hcal$ is $(\Aut,N,\eta/2)$-nice for $w$.
\end{lemmma}

\begin{proof}
By the assumption on the untwisted Lie rank of members of $\Hcal$, the assumptions of Lemma \ref{similarGroupsLem} with $C:=1$ are satisfied; more precisely, fixing an element $\overline{\Omega}\in\Hcal$:

\begin{itemize}
\item Since the untwisted Lie rank of $\overline{\Omega}$ is at least $5$, by the observations before Lemma \ref{explicitLem}, considering the associated isometry group $I\in\G$, $\overline{\Omega}$ is a characteristic section of $I$ such that every automorphism of $\overline{\Omega}$ \enquote{comes from} an automorphism from $\A(I)\leq\Aut(I)$.
\item Furthermore, since the untwisted Lie rank of $\overline{\Omega}$ is at least $2(1+d-\eta)/\eta=\epsilon^{-1}$, by Lemma \ref{explicitLem}, we also have $|I|\leq|\overline{\Omega}|^{1+\epsilon}$.
\end{itemize}

Hence we are done by an application of Lemma \ref{similarGroupsLem}.
\end{proof}

We now give the aforementioned theorem to which Theorem \ref{simpleGroupsTheo} reduces:

\begin{theoremm}\label{alternativeTheo}
Let $w$ be a reduced word of length $l\geq1$ in $d$ distinct variables, and let $M$ be as in Notation \ref{mNot}. Then the following hold:

\begin{enumerate}
\item The class of finite symmetric groups is $(\Aut,\lceil256l^{16}\e^{16Md-2}\rceil!,1/(8M))$-nice for $w$.
\item The class of isometry groups associated with the finite simple groups of Lie type of untwisted Lie rank at least $72(d+1)^2l^2$ is $(\A,1,1/(36(d+1)l^2))$-nice for $w$.
\end{enumerate}
\end{theoremm}

Let us actually derive Theorem \ref{simpleGroupsTheo} from this.

\begin{proof}[Proof of Theorem \ref{simpleGroupsTheo} using Theorem \ref{alternativeTheo}]
For (1): Applying Lemma \ref{altSymLem}, we get from Theorem \ref{alternativeTheo}(1) that the class of finite alternating groups has the following niceness tuple for $w$:

\[
(\Aut,\max\{\lceil256l^{16}\e^{16Md-2}\rceil!,2^{1+2(1+d-\eta)/\eta}\},1/(16M)),
\]

where $\eta=1/(8M)$. It is not difficult to check that the second term in the maximum expression in the second entry of the tuple is smaller than the first term, and we are done.

For (2): By Lemma \ref{omegaILem} and Theorem \ref{alternativeTheo}(2), we only need to check that for $\eta=1/(36(d+1)l^2)$, we have $2(1+d-\eta)/\eta\leq 72(d+1)^2l^2$, which is elementary.
\end{proof}

\subsection{First part of the proof of Theorem \ref{alternativeTheo}: Symmetric groups and isometry groups other than general linear groups}\label{subsec3P3}

We now turn to the proof of Theorem \ref{alternativeTheo}, which as mentioned before, is a modification of an argument by Larsen and Shalev from \cite{LS12a}. Let us first make some general observations which will be used in the proof.

Note that each of the abstract groups $G$ with which Theorem \ref{alternativeTheo} deals can actually be viewed as a permutation group, acting on a set $\Delta$, in a natural way: each symmetric group $\Sym_n$ through its natural action on the set $\{1,\ldots,n\}$, and each isometry group through its action on the corresponding vector space. Larsen and Shalev also exploited this fact, and their argument consisted essentially in investigating to what extent a relation of the form $w(g_1,\ldots,g_d)=g$ for $g\in G$ fixed imposes restrictions on $g_1,\ldots,g_d\in G$ when viewed as maps $\Delta\rightarrow\Delta$.

In our setting, this gets more complicated because we are actually considering not a word equation in $g_1,\ldots,g_d$, but a word equation in various images of the $g_i$ under fixed automorphisms of $G$ from some subgroup $\A(G)\leq\Aut(G)$. Hence it would be useful if, from some single piece of mapping information of the form $\alpha(g_i)(x)=y$, $\alpha\in\A(G)$ and $x,y\in\Delta$, we could derive such a condition on $g_i$ itself. It turns out that this is actually possible for the $G$ with which we are concerned except for the case $G=\GL_n(q)$, which will require some separate treatment.

In this subsection, we deal with the $G$ not isomorphic with any $\GL_n(q)$.

\begin{nottation}\label{tNot}
We introduce the following notation:

\begin{enumerate}
\item As $G=\Sym_n$ with $n\geq 7$ is complete, for each automorphism $\alpha$ of $G$, there is a unique $\sigma\in G$ such that $\alpha=\conj(\sigma)$. We set $\t(\alpha):=\sigma^{-1}\in\Sym_n=\Sym_{\Delta}$ with $\Delta=\{1,\ldots,n\}$.
\item Let $G=I$ be the isometry group of either a perfect symmetric, perfect anti-symmetric or perfect Hermitian pairing of a finite vector space $\Delta=\IF_{q^e}^m$, where $e=1$ in the symmetric and anti-symmetric case and $e=2$ in the Hermitian case. By the definition of $\A(I)\leq\Aut(I)$ above, it is clear that every element $\alpha\in\A(I)$ is of the form $\conj(U)\circ\aut(\sigma)$, where $U\in I$ and $\aut(\sigma)$ is a field automorphism of $I$, induced by an automorphism $\sigma$ of $\IF_{q^e}$. $\sigma$ also induces a permutation $\perm(\sigma)$ on $\Delta$, namely the map $\Delta\rightarrow \Delta,(x_1,\ldots,x_m)\mapsto(\sigma(x_1),\ldots,\sigma(x_m))$. We set $\t(\alpha):=(U\circ\perm(\sigma))^{-1}\in\Sym_{\Delta}$.
\end{enumerate}
\end{nottation}

The point behind Notation \ref{tNot} is that in each of the cases considered, the automorphism $\alpha$ of $G$ can be seen as the restriction of the conjugation by $\t(\alpha)^{-1}\in\Sym_{\Delta}$ to the subgroup $G$ of $\Sym_{\Delta}$. Hence the following is clear:

\begin{lemmma}\label{mappingLem}
Let $G$ be $\Sym_n$ for some $n\geq 7$ resp.~an isometry group as in Notation \ref{tNot}(2). Then for every $\alpha\in\Aut(G)$ (resp. $\alpha\in\A(G)$), for every $g\in G$, and for all $x,y\in\Delta$, the set on which $G$ acts naturally, we have $\alpha(g)(x)=y$ if and only if $g(\t(\alpha)(x))=\t(\alpha)(y)$.\qed
\end{lemmma}

At last, we are now ready to discuss the proofs of Theorem \ref{alternativeTheo}(1) and of Theorem \ref{alternativeTheo}(2) except for general linear groups; we will present these proofs one after the other.

\begin{proof}[Proof of Theorem \ref{alternativeTheo}(1)]
Let $G=\Sym_n$ with $n\geq 256l^{16}\e^{16Md-2}$. We need to show that $\P_w(G)\leq |G|^{d-1/(8M)}$. By Lemma \ref{mainLem}(2), we know that $\P_w(G)=\P_w(G,1)$, so we only need to bound the maximum size of the fiber $\Phi$ of $1=\id$ under an automorphic word map on $G$. Hence fix automorphisms $\alpha_1,\ldots,\alpha_l\in\Aut(G)\cong G$, and, as usual, write $w=x_1^{\epsilon_1}\cdots x_l^{\epsilon_l}$ with $\epsilon_i\in\{\pm1\}$, $x_1,\ldots,x_l\in\{X_1,\ldots,X_d\}$ and $\iota:\{1,\ldots,l\}\rightarrow\{1,\ldots,d\}$ such that $x_i=X_{\iota(i)}$.

We associate with each fixed $d$-tuple $\vec{g}=(g_1,\ldots,g_d)\in G^d$ a certain metric $d_{\vec{g}}^{(\alpha_1,\ldots,\alpha_l)}$ on $\Delta=\{1,\ldots,n\}$, as follows: For $y,z\in\Delta$, if $z$ can be obtained from $y$ through a finite number of applications of permutations on $\Delta$, each of one of the two forms

\begin{itemize}
\item $\alpha_i(g_j)^{\pm 1}$, where $i\in\{1,\ldots,l\}$ and $j\in\{1,\ldots d\}$, or
\item $\t(\alpha_i)^{\pm 1}$, where $i\in\{1,\ldots,l\}$,
\end{itemize}

then $d_{\vec{g}}^{(\alpha_1,\ldots,\alpha_l)}(y,z)$ is defined as the smallest number of such function applications which it takes to pass from $y$ to $z$. If, on the other hand, $z$ cannot be obtained from $y$ in this way, we set $d_{\vec{g}}^{(\alpha_1,\ldots,\alpha_l)}(y,z):=n$. It is easy to check that $d_{\vec{g}}^{(\alpha_1,\ldots,\alpha_l)}$ really is a metric on $\Delta$.

We call elements $y,z\in\Delta$ \emph{independent} if and only if there do not exist $\alpha,\beta\in\{\alpha_1,\ldots,\alpha_l\}$ such that $\t(\alpha)(y)=\t(\beta)(z)$, and else we call them \emph{dependent}. Furthermore, we define $u_j:=x_{l-j+1}^{\epsilon_{l-j+1}}\cdots x_l^{\epsilon_l}$ for $j=0,\ldots,l$ (the terminal segments of $w$), so that $x_{l-j}^{\epsilon_{l-j}}u_j=u_{j+1}$. Finally, we set $\nu_j:=(u_j)_G^{(\alpha_{l-j+1},\ldots,\alpha_l)}(\vec{g})\in G$, $j=0,\ldots,l$, and $L:=\lfloor n/(4M)\rfloor$.

Now let $\vec{z}=(z_1,\ldots,z_L)$ denote an ordered $L$-tuple of elements $\Delta$. We consider two sets $X$ and $X'$:

\begin{align}\label{xEq}
X := \{(\vec{g},\vec{z})\in G^d\times\Delta^L\mid &\forall i\not=j: d_{\vec{g}}^{(\alpha_1,\ldots,\alpha_l)}(z_i,z_j)>2l+2 \text{ and } \notag \\
&\forall i: |\{z_i,\nu_1(z_i),\ldots,\nu_l(z_i)\}|\leq l\},
\end{align}

\begin{align}\label{xPrimeEq}
X' := \{(\vec{g},\vec{z})\in G^d\times\Delta^L\mid &\forall i\not=j: d_{\vec{g}}^{(\alpha_1,\ldots,\alpha_l)}(z_i,z_j)>2l+2 \text{ and } \notag \\
&\forall i: \exists j_1,j_2\in\{0,\ldots, l\}: (j_1\not=j_2 \text{, and }\nu_{j_1}(z_i) \text{ and } \nu_{j_2}(z_i) \notag \\
&\text{ are dependent})\}.
\end{align}

Note that the second condition, $|\{z_i,\nu_1(z_i),\ldots,\nu_l(z_i)\}|\leq l$, in Formula (\ref{xEq}) just means that two of the elements $z_i,\nu_1(z_i),\ldots,\nu_l(z_i)$ are equal, which is a stronger condition than the second condition in Formula (\ref{xPrimeEq}). Hence $X\subseteq X'$. Our goal is to determine an upper bound on $|X|$, and to this end, we bound $|X'|$.

We begin by fixing two $L$-tuples $(a_1,\ldots,a_L)$ and $(b_1,\ldots,b_L)$ of non-negative integers such that $a_i<b_i\leq l$ for all $i$, as well as two $L$-tuples $(\gamma_1,\ldots,\gamma_L)$ and $(\delta_1,\ldots,\delta_L)$ with entries from the set $\{\alpha_1,\ldots,\alpha_l\}$. There are fewer than $l^{4L}$ choices for this.

For each such choice, we count only the elements $(\vec{g},\vec{z})$ of $X'$ such that for all $i\leq L$, the elements $z_i,\nu_1(z_i),\ldots,\nu_{b_i-1}(z_i)\in\Delta$ are pairwise independent, while the dependence relation $\t(\gamma_i)(\nu_{b_i}(z_i))=\t(\delta_i)(\nu_{a_i}(z_i))$ holds. There are fewer than $n^{b_1+\cdots+b_L}$ ways of choosing ordered tuples $Z_1,\ldots,Z_L$ with entries from $\Delta$ and of length $b_1,\ldots,b_L$ respectively such that the entries of each tuple are pairwise independent. For fixed $(Z_1,\ldots,Z_L)$, we count only elements of $X'_{Z_1,\ldots,Z_L}$, i.e., only elements of $X'$ as specified above such that additionally, $(z_i,\nu_1(z_i),\ldots,\nu_{b_i-1}(z_i))=Z_i$ for each $i=1,\ldots,L$.

Now the distance condition in Formula (\ref{xPrimeEq}) implies that if any coordinate of $Z_i$ is in dependence with any coordinate of $Z_j$ for $i\not=j$, then $X'_{Z_1,\ldots,Z_L}=\emptyset$. We may therefore assume that coordinates of $Z_i$ and $Z_j$, $i\not=j$, are always independent, a feature which we call the \emph{inter-independence of the $Z_i$}.

Note that for each $i\in\{1,\ldots,L\}$ and each $j\in\{1,\ldots,b_i\}$, we get the following condition on one of the functions $g_1,\ldots,g_d:\Delta\rightarrow\Delta$:

\begin{itemize}
\item if $\epsilon_{l-j+1}=+1$ and $\iota(l-j+1)=k$, then $\alpha_{l-j+1}(g_k)(\nu_{j-1}(z_i))=\nu_j(z_i)$, or equivalently (by Lemma \ref{mappingLem}) $g_k(\t(\alpha_{l-j+1})(\nu_{j-1}(z_i)))=\t(\alpha_{l-j+1})(\nu_j(z_i))$.
\item if $\epsilon_{l-j+1}=-1$ and $\iota(l-j+1)=k$, then $\alpha_{l-j+1}(g_k)(\nu_j(z_i))=\nu_{j-1}(z_i)$, or equivalently $g_k(\t(\alpha_{l-j+1})(\nu_j(z_i)))=\t(\alpha_{l-j+1})(\nu_{j-1}(z_i))$.
\end{itemize}

Let us introduce some terminology for conditions of the form $f(x)=y$, where $f$ is a variable standing for a function $\Delta\rightarrow\Delta$ and $x,y\in\Delta$ are fixed. We call $x$ the \emph{argument} and $y$ the \emph{image in the condition $f(x)=y$}. Call two such conditions $f(x_1)=y_1$ and $g(x_2)=y_2$ \emph{independent} if and only if either $f$ and $g$ are distinct variables or $f=g$ and $x_1\not=x_2$. Two conditions that are not independent are called \emph{dependent}. Finally, the conditions $f(x_1)=y_1$ and $g(x_2)=y_2$ are called \emph{contradictory} if and only if $f=g$, $x_1=x_2$ and $y_1\not=y_2$.

Equipped with this terminology, we note that for fixed $i$, either are two of the $b_i$ conditions on the $g_k$ derived above contradictory (so that $X'_{Z_1,\ldots,Z_L}=\emptyset$ in this case as well), or the conditions are pairwise independent. To see this, note that if the conditions are not pairwise independent, then since we are assuming that $z_i,\nu_1(z_i),\ldots,\nu_{b_i-1}(z_i)$ are pairwise independent elements of $\Delta$ (in the sense defined before Formula (\ref{xEq})), the existing pair of dependent conditions is unique, and one of the two conditions has an image of the form $\t(\alpha)(\nu_j(z_i))$ with $1\leq j\leq b_i-1$, and the other condition is $g_k(\t(\alpha_{l-b_i+1})(\nu_{b_i}(z_i)))=\t(\alpha_{l-b_i+1})(\nu_{b_i-1}(z_i))$. Now since no two consecutive terms in the sequence $x_l^{\epsilon_l},\ldots,x_1^{\epsilon_1}$ are mutually inverse in the corresponding free group, we must have $j<b_i-1$, but this, again by the pairwise independence of $z_i,\nu_1(z_i),\ldots,\nu_{b_i-1}(z_i)$, shows that the images in the two conditions cannot be equal, and so the conditions are contradictory, as we wanted to show.

We may thus assume that for fixed $i$, the $b_i$ conditions listed above are pairwise independent, and the inter-independence of the $Z_i$ then guarantees us that actually all the $b_1+\cdots+b_L$ conditions described above are pairwise independent. As the number of elements of $X'_{Z_1,\ldots,Z_l}$ is bounded from above by the number of $d$-tuples of functions $\Delta\rightarrow\Delta$ satisfying all the $b_1+\cdots+b_L$ conditions above, we conclude that $|X'_{Z_1,\ldots,Z_l}|\leq n^{dn-b_1-\cdots-b_L}$. It follows that

\begin{equation}\label{xUpperBoundEq}
|X|\leq|X'|\leq l^{4L}n^{dn}.
\end{equation}

To get an upper bound on the size of $\Phi$, the fiber of $\id$ under $w_G^{(\alpha_1,\ldots,\alpha_l)}$, from this, note that for each $\vec{g}\in G^d$ lying in that fiber, we have

\[
(\{\vec{g}\}\times\Delta^L)\cap X = \{(\vec{g},\vec{z})\in\{\vec{g}\}\times\Delta^L\mid \forall i\not=j: d_{\vec{g}}^{(\alpha_1,\ldots,\alpha_L)}(z_i,z_j)>2l+2\}.
\]

Now the ball $\B_{2l+2}(z)$ of radius $2l+2$ with respect to the metric $d_{\vec{g}}^{(\alpha_1,\ldots,\alpha_l)}$ around any $z\in\Delta$ has, by definition of $M$, cardinality at most $M$. Furthermore, by definition of $L$, $LM<n/4$. Hence if we select $z_1,\ldots,z_L\in\Delta$ iteratively so that for each $j=1,\ldots,L$,

\[
z_j\in\{z\in\Delta\mid \forall i<j: d_{\vec{g}}^{(\alpha_1,\ldots,\alpha_l)}(z_j,z_i)>2l+2\},
\]

then the number of possibilities for $z_j$ is at least

\[
|\Delta\setminus\bigcup_{i=1}^{j-1}{\B_{2l+2}(z_i)}|\geq n/4.
\]

It follows that

\[
|(\{\vec{g}\}\times\Delta^L)\cap X|\geq (n/4)^L.
\]

Hence we also have a lower bound on the cardinality of $X$:

\begin{equation}\label{xLowerBoundEq}
|X|\geq |(\Phi\times\Delta^L)\cap X|\geq |\Phi|\cdot (n/4)^L.
\end{equation}

Combining Formula (\ref{xLowerBoundEq}) with the upper bound on $|X|$ from Formula (\ref{xUpperBoundEq}), we conclude that

\begin{align}\label{boundOnFEq}
|\Phi| &\leq (n/4)^{-L}l^{4L}n^{dn} \notag \\
&= (\sqrt{2}l)^{4L}n^{dn-L} \notag \\
&\leq (\sqrt{2}l)^{n/M}\cdot n^{1+(d-1/(4M))n}.
\end{align}

From the explicit Stirling-like bound $n!\geq (n/\e)^n$ (which, as noted in \cite{Tao10a}, is an immediate consequence of the Taylor expansion of the exponential function), it is clear from Formula (\ref{boundOnFEq}) that $|\Phi|\leq |G|^{d-1/(8M)}$ as long as

\[
(\sqrt{2}l)^{n/M}\cdot n^{1+(d-1/(4M))n} \leq (n/\e)^{n(d-1/(8M))},
\]

which is equivalent to

\begin{equation}\label{transformedEq}
(\sqrt{2}l)^{1/M}\cdot \e^{d-1/(8M)}\cdot n^{1/n+d-1/(4M)} \leq n^{d-1/(8M)}.
\end{equation}

Now note that our assumption $n\geq 256l^{16}\e^{16Md-2}=(2l^2)^8e^{16Md-2}$ is equivalent to $(\sqrt{2}l)^{1/M}\cdot \e^{d-1/(8M)}\leq n^{1/(16M)}$. Hence by Formula (\ref{transformedEq}), all that we need to finish the proof is to verify that $n^{1/n+d-3/(16M)}\leq n^{d-1/(8M)}$, which is equivalent to $n\geq 16M$, and this is certainly true by our assumption on $n$.
\end{proof}

For the other proof, we require the following lemma, which is essentially \cite[Lemma 3.2]{LS12a}:

\begin{lemmma}\label{stabilizerLem}
Let $G$ be the isometry group, acting naturally on a finite vector space $\Delta$, associated with a classical finite simple group of Lie type $S=X(q)$. Set $E:=\IF_q$, and denote by $F$ the finite field such that $\Delta$ is an $F$-vector space (recall that either $F=E$ or, in the Hermitian case, $F$ is a quadratic extension of $E$). Set $n:=\dim_E(\Delta)$, and let $v_1,\ldots,v_k$ be $E$-linearly independent vectors in $V$ such that $n\geq 2k+2$. Then $|\Stab_G(v_1,\ldots,v_k)|\leq q^{k^2+k-kn}\cdot|G|$.\qed
\end{lemmma}

\begin{proof}[Proof of Theorem \ref{alternativeTheo}(2) except for general linear groups]
Let $G$ be the isometry group of either a perfect symmetric, perfect anti-symmetric or perfect Hermitian pairing on a finite $F$-vector space $\Delta$. In the first two cases, set $E:=F$, and in the Hermitian case, let $E$ be the unique subfield of $F$ such that $[F:E]=2$. Furthermore, set $q:=|E|$ and $n:=\dim_E(\Delta)$ as well as $m:=\dim_F(\Delta)=n/e$ (with $e$ as in Notation \ref{tNot}(2)), so that w.l.o.g.~$\Delta=\IF_{q^e}^m$ and Notation \ref{tNot}(2) is applicable. Finally, fix $\alpha_1,\ldots,\alpha_l\in\A(G)$.

Under these assumptions, we will actually show something stronger than what is asserted in Theorem \ref{alternativeTheo}(2) for all isometry groups (including the general linear groups), namely that if $n\geq 216l^2$, then the size of the fiber $\Phi$ of $1_G=\id$ under $w_G^{(\alpha_1,\ldots,\alpha_l)}$ is at most $|G|^{d-1/(72l^2)}$ (note that it is sufficient to consider that fiber by Lemma \ref{mainLem}(2), as $\A(G)$ contains $\Inn(G)$). As before, the argument is a modification of a proof of Larsen and Shalev, namely of \cite[proof of Proposition 3.3]{LS12a}. Compared to their situation, we have the advantage that we only need to consider the fiber of $\id$, not of any isometry with an eigenvalue of multiplicity at least $n/3$, so that some parts of the construction even get simpler, while others get more complicated to make them still work for automorphic word maps.

Let $\vec{g}=(g_1,\ldots,g_d)$ denote a $d$-tuple of elements of $G$. We define $u_j$ and $\nu_j$ by the same formulas as in the proof of Theorem \ref{alternativeTheo}(1) above. Furthermore, we set $L:=\lfloor n/(9l^2)\rfloor$ and let $\vec{z}=(z_1,\ldots,z_L)$ denote an $L$-tuple of elements of $\Delta$. We define the lexicographic order ${}\prec{}$ on the set $\{1,\ldots,L\}\times\{0,\ldots,l\}$ through $(i',j')\prec(i,j)$ if and only if $i'<i$, or $i=i'$ and $j'<j$. Finally, we define

\begin{align}\label{xDefEq}
X:=\{(\vec{g},\vec{z})\in G^d\times\Delta^L\mid \forall i:(&z_i\notin\Span_{(i',j')\prec(i,0),k=1,\ldots,l}{\t(\alpha_k)(\nu_{j'}(z_{i'}))} \text{ and } \notag \\
&\nu_l(v_i)\in\Span_{(i',j')\prec(i,l),k=1,\ldots,l}{\t(\alpha_k)(\nu_{j'}(z_{i'}))})\},
\end{align}

where here and in the rest of this proof, for a subset $A\subseteq\Delta$, $\Span{A}$ denotes the $E$-span of $A$ inside $\Delta$.

For each $(\vec{g},\vec{z})\in X$, we define $b_i$ to be the smallest positive integer such that

\begin{equation}\label{spanEq}
\nu_{b_i}(z_i)\in\Span_{(i',j')\prec(i,b_i),k=1,\ldots,l}{\t(\alpha_k)(\nu_{j'}(z_{i'}))}.
\end{equation}

Note that $1\leq b_i\leq l$, and so $b_1+\cdots+b_L\leq lL$. We make Formula (\ref{spanEq}) more explicit by fixing $a_{i,i',j',k}\in E=\IF_q$ such that

\begin{equation}\label{spanEq2}
\nu_{b_i}(z_i)=\sum_{(i',j')\prec(i,b_i),k=1,\ldots,l}{a_{i,i',j',k}\t(\alpha_k)(\nu_{j'}(z_{i'}))}.
\end{equation}

There are fewer than $q^{l^2L^2}l^L$ ways in which the $a_{i,i',j',k}$ and $b_i$ can be chosen:

\begin{itemize}
\item precisely $l^L$ ways for the choice of $b_i$,
\item and less than the following number of ways for the choice of the scalars $a_{i,i',j',k}$ from $E=\IF_q$:

\[
q^{l(b_1+(l+b_2)+(2l+b_3)+\cdots+((L-1)l+b_L))}\leq q^{l(lL+l\cdot L(L-1)/2)}=q^{l^2L(1+(L-1)/2)}<q^{l^2L^2}.
\]
\end{itemize}

Furthermore, there are fewer than $q^{n(b_1+\cdots+b_L)}$ possibilities for the sequence of sequences

\[
\overline{z}=(z_1,\ldots,\nu_{b_1-1}(z_1); z_2,\ldots,\nu_{b_2-1}(z_2); \ldots; z_l,\ldots,\nu_{b_L-1}(z_L))
\]

such that none of the vectors in the sequence lies in the $E$-span of all the vectors obtained by applying one of the $\t(\alpha_k)$, $k=1,\ldots,l$, to one of the previous vectors in the sequence.

We estimate the number of elements $(\vec{g},\vec{z})$ of $X$ for fixed choices of $a_{i,i',j',k}$, $b_i$ and $\overline{z}$. Note that $\vec{z}$ is already fixed now as a part of $\overline{z}$, so we need to bound the number of matching $\vec{g}=(g_1,\ldots,g_d)\in G^d$. Say $\alpha_k=\conj(U_k)\circ\aut(\sigma_k)$ for $k=1,\ldots,l$, where $U_k\in G$ and $\sigma_k$ is an automorphism of $F=\IF_{q^e}$. Note that by the definition of $\t(\alpha_k)$ in Notation \ref{tNot}(2), the map $\t(\alpha_k):\Delta\rightarrow\Delta$ is $F$-semilinear (in the sense of \cite[bottom of p.~9]{KL90a}); more precisely, we have, for all $v,w\in\Delta$ and all $\lambda\in F=\IF_{q^e}$: $\t(\alpha_k)(v+w)=\t(\alpha_k)(v)+\t(\alpha_k)(w)$ and $\t(\alpha_k)(\lambda\cdot v)=\sigma_k^{-1}(\lambda)\cdot\t(\alpha_k)(v)$. Also, note that if $\lambda\in E$, then $\sigma_k^{-1}(\lambda)\in E$ as well.

We get the following $b_1+\cdots+b_L$ conditions on the $g_k$:

For each $i=1,\ldots,L$:

\begin{itemize}
\item for each $j=1,\ldots,b_i-1$:
\begin{itemize}
\item if $\epsilon_{l-j+1}=+1$ and $\iota(l-j+1)=k$: $\alpha_{l-j+1}(g_k)(\nu_{j-1}(z_i))=\nu_j(z_i)$, which by Lemma \ref{mappingLem} is equivalent to $g_k(\t(\alpha_{l-j+1})(\nu_{j-1}(z_i)))=\t(\alpha_{l-j+1})(\nu_j(z_i))$.
\item if $\epsilon_{l-j+1}=-1$ and $\iota(l-j+1)=k$: $\alpha_{l-j+1}(g_k)(\nu_j(z_i))=\nu_{j-1}(z_i)$, which by Lemma \ref{mappingLem} is equivalent to $g_k(\t(\alpha_{l-j+1})(\nu_j(z_i)))=\t(\alpha_{l-j+1})(\nu_{j-1}(z_i))$.
\end{itemize}
\item if $\epsilon_{l-b_i+1}=+1$ and $\iota(l-b_i+1)=k$, then

\[
\alpha_{l-b_i+1}(g_k)(\nu_{b_i-1}(z_i))=\sum_{(i',j')\prec(i,b_i),o=1,\ldots,l}{a_{i,i',j',o}\nu_{j'}(z_{i'})},
\]

which by Lemma \ref{mappingLem} and the semilinearity of the $\t(\alpha_k)$ is equivalent to

\[
g_k(\t(\alpha_{l-b_i+1})(\nu_{b_i-1}(z_i)))=\sum_{(i',j')\prec(i,b_i),o=1,\ldots,l}{\sigma_{l-b_i+1}^{-1}(a_{i,i',j',o})\t(\alpha_{l-b_i+1})(\nu_{j'}(z_{i'}))}.
\]

\item if $\epsilon_{l-b_i+1}=-1$ and $\iota(l-b_i+1)=k$, then

\[
\alpha_{l-b_i+1}(g_k)(\sum_{(i',j')\prec(i,b_i),o=1,\ldots,l}{a_{i,i',j',o}\nu_{j'}(z_{i'})})=\nu_{b_i-1}(z_i),
\]

which is equivalent to

\[
g_k(\sum_{(i',j')\prec(i,b_i),o=1,\ldots,l}{\sigma_{l-b_i+1}^{-1}(a_{i,i',j',o})\t(\alpha_{l-b_i+1})}(\nu_{j'}(z_{i'})))=\t(\alpha_{l-b_i+1})(\nu_{b_i-1}(z_i)).
\]
\end{itemize}

Like in the proof of Theorem \ref{alternativeTheo}(1), we now argue that this system of conditions of the form $g_k(v)=w$ is either contradictory (i.e., not satisfiable for any choice of the $g_k$ in $\End_F(\Delta)$) or the conditions are independent, meaning here that for each $k$, the set of all vectors appearing as arguments in one of the conditions concerning $g_k$ is $E$-linearly independent.

Indeed, assume that for some $k$, the set of argument vectors for $g_k$ is $E$-linearly dependent. Note that the lexicographical order $\prec$ which we defined on $\{1,\ldots,L\}\times\{0,\ldots,l\}$ also induces a linear order on the conditions involving the variable $g_k$, as each such condition is by definition associated with a pair $(i,j)\in\{1,\ldots,L\}\times\{0,\ldots,l\}$ in an injective way (for a condition as described in the last two bullet points above, this pair is $(i,b_i)$). By means of this linear order, list the conditions involving $g_k$ as follows: $g_k(v_1)=w_1, g_k(v_2)=w_2,\ldots, g_k(v_{t_k})=w_{t_k}$. Since the set $\{v_1,\ldots,v_{t_k}\}$ is $E$-linearly dependent by assumption, there exists $u\in\{2,\ldots,t_k\}$ such that $v_u\in\Span\{v_1,\ldots,v_{u-1}\}$. Note that if the system of conditions is satisfiable through a suitable choice of $g_1,\ldots,g_d\in\End_F(\Delta)$, then this implies that likewise $w_u\in\Span\{w_1,\ldots,w_{u-1}\}$. We will now argue that this is not the case.

By choice of $\overline{z}$, the assumption that $v_u\in\Span\{v_1,\ldots,v_{u-1}\}$ implies that $g_k(v_u)=w_u$ must be a condition as described in the third bullet point above, with $v_u=\sum_{(i',j')\prec(i,b_i),o=1,\ldots,l}{\sigma_{l-b_i+1}^{-1}(a_{i,i',j',o})\t(\alpha_{l-b_i+1})}(\nu_{j'}(z_{i'}))$ and $w_u=\t(\alpha_{l-b_i+1})(\nu_{b_i-1}(z_i))$ (and thus $\epsilon_{l-b_i+1}=-1$). Using that no two consecutive terms in the sequence $x_l^{\epsilon_l},\ldots,x_1^{\epsilon_1}$ are mutually inverse in the corresponding free group, we get that none of the conditions $g_k(v_1)=w_1, \ldots, g_k(v_{u-1})=w_{u-1}$ is associated with the pair $(i,b_i-1)$, and the assertion that $w_u\notin\Span\{w_1,\ldots,w_{u-1}\}$ now follows again by choice of $\overline{z}$.

Hence we may assume w.l.o.g.~that the above described $b_1+\cdots+b_L$ conditions on the $g_k$ are independent, so that by Lemma \ref{stabilizerLem} and the convexity of the function $r\mapsto r^2+r$, we see that there are no more than

\[
q^{(b_1+\cdots+b_L)^2+(b_1+\cdots+b_L)-(b_1+\cdots+b_L)n}|G|^d
\]

elements of $X$, subject to the choices of $a_{i,i',j,k}$, $b_i$ and $\overline{z}$. Hence

\begin{equation}\label{xUpperBoundEq2}
|X|\leq l^Lq^{l^2L^2+2(b_1+\cdots+b_L)^2}|G|^d\leq l^Lq^{l^2L^2+2l^2L^2}|G|^d=l^Lq^{3l^2L^2}|G|^d.
\end{equation}

On the other hand, if $\vec{g}\in G^d$ lies in $\Phi$, then for all $\vec{z}\in\Delta^L$, $(\vec{g},\vec{z})$ is an element of $X$ if and only if for all $i=1,\ldots,L$, the condition

\begin{equation}\label{spanEq3}
z_i\notin\Span_{(i',j')\prec(i,0),k=1,\ldots,l}{\t(\alpha_k)(\nu_{j'}(z_{i'}))}
\end{equation}

is satisfied. Now for each $i$, the span on the RHS of Formula (\ref{spanEq3}) has $E$-dimension less than $l^2L\leq n/9\leq n-1$ and thus is a proper $E$-subspace of $\Delta$. It follows that in each step of iteratively fixing an $L$-tuple $(z_1,\ldots,z_L)\in\Delta^L$ according to Formula (\ref{spanEq3}), we have at least $q^{n-1}$ many choices for $z_i$. Hence the number of pairs $(\vec{g},\vec{z})\in X$ with $\vec{g}\in\Phi$ fixed is at least $q^{L(n-1)}$, and it follows that

\begin{equation}\label{xLowerBoundEq2}
|X|\geq |\Phi|\cdot q^{L(n-1)}.
\end{equation}

Combining Formulas (\ref{xUpperBoundEq2}) and (\ref{xLowerBoundEq2}), we get that

\begin{align*}
|\Phi| &\leq l^Lq^{3l^2L^2-L(n-1)}|G|^d \leq l^{n/(9l^2)}q^{3l^2\cdot n^2/(81l^4)-n/(9l^2)\cdot n/2}|G|^d \\
&\leq q^n\cdot q^{n^2(1/(27l^2)-1/(18l^2))}|G|^d = q^{n-n^2/(54l^2)}|G|^d = (q^{n^2})^{1/n-1/(54l^2)}|G|^d \\
&\leq (q^{n^2})^{1/(216l^2)-1/(54l^2)}|G|^d = (q^{n^2})^{-1/(72l^2)}|G|^d \leq |G|^{d-1/(72l^2)},
\end{align*}

where in the last step, we used that $|G|\leq q^{n^2}$, which in the symmetric and anti-symmetric cases is trivial since $G\leq\GL_n(q)$ then, and in the Hermitian case, it follows from $|G|=|\GU_n(q)|=q^{n(n-1)/2}(q^n-(-1)^n)(q^{n-1}-(-1)^{n-1})\cdots(q^2-1)(q+1)$, see, for example, \cite[p.~x]{CCNPW85a}.
\end{proof}

\subsection{Second part of the proof of Theorem \ref{alternativeTheo}: General linear groups}\label{subsec3P4}

As mentioned before, for the general linear groups $G=\GL_n(q)$, the argument used for the other isometry groups from Theorem \ref{alternativeTheo}(2) needs to be modified. This is because the automorphisms of $G$ which can be written as $\conj(U)\circ\aut(\sigma)$ for some $U\in G$ and $\sigma\in\Aut(\IF_q)$ only form an index $2$ subgroup, hitherto denoted by $\B(\GL_n(q))=\B(G)$, in $\A(G)$. A representative for the other coset of $\B(G)$ in $\A(G)$ is the inverse-transpose automorphism $\tau:U\mapsto (U^{-1})^t=(U^t)^{-1}$. This also means that it is not possible in general to rewrite a condition of the form $\alpha(g)(v)=w$ with $\alpha\in\A(G)$ equivalently into one of the form $g(\t(\alpha)(v))=\t(\alpha)(w)$ as before. However, it is easy to see that we can at least rewrite each such condition equivalently into one of two possible forms:

\begin{lemmma}\label{rewritingLem}
Let $G=\GL_n(q)$ for some $n\in\IN^+$ and prime power $q$, and let $\Delta:=\IF_q^n$, an $\IF_q$-vector space on which $G$ acts naturally. Furthermore, let $\alpha\in\A(G)$, $g\in G$ and $x,y\in\Delta$. Then the following hold:

\begin{enumerate}
\item If $\alpha\in\B(G)$, say $\alpha=\conj(U)\circ\sigma$, then setting $\t(\alpha):=(U\circ\perm(\sigma))^{-1}\in\Sym_{\Delta}$ just as in Notation \ref{tNot}(2), we have that $\alpha(g)x=y$ is equivalent to $g\t(\alpha)(x)=\t(\alpha)(y)$.
\item If $\alpha\in\A(G)\setminus\B(G)$, say $\alpha=\beta\circ\tau$ with $\beta=\conj(U)\circ\sigma$, then $\alpha(g)x=y$ is equivalent to $g^t\t(\beta)(y)=\t(\beta)(x)$.
\end{enumerate}
\end{lemmma}

\begin{proof}
The argument for point (1) is like the one for Lemma \ref{mappingLem}: that $\alpha$ can be viewed as the restriction of the inner automorphism $\conj(\t(\alpha)^{-1}):\Sym_{\Delta}\rightarrow\Sym_{\Delta}$ to $G\leq\Sym_{\Delta}$.

As for point (2), note that

\[
\alpha(g)x=y \Leftrightarrow \beta((g^t)^{-1})x=y \Leftrightarrow \beta(g^t)^{-1}x=y \Leftrightarrow \beta(g^t)y=x \Leftrightarrow g^t\t(\beta)(y)=\t(\beta)(x),
\]

as required.
\end{proof}

In view of this, the following Lemma will act as a substitute for Lemma \ref{stabilizerLem}:

\begin{lemmma}\label{substituteLem}
Let $n\in\IN^+$, $q$ a prime power, $r_1,r_2\in\IN$ with $r_1,r_2\leq n$. Let $v_1^{(1)},\ldots,v_{r_1}^{(1)},w_1^{(1)},\ldots,w_{r_1}^{(1)},v_1^{(2)},\ldots,v_{r_2}^{(2)},w_1^{(2)},\ldots,w_{r_2}^{(2)}\in\IF_q^n$ such that $v_1^{(1)},\ldots,v_{r_1}^{(1)}$ are $\IF_q$-linearly independent and $v_1^{(2)},\ldots,v_{r_2}^{(2)}$ are $\IF_q$-linearly independent. Then the number of $g\in\Mat_n(q)$ such that $gv_i^{(1)}=w_i^{(1)}$ for $i=1,\ldots,r_1$ and $g^tv_j^{(2)}=w_j^{(2)}$ for $j=1,\ldots,r_2$ is at most $q^{n^2-(r_1+r_2)n+r_1r_2}$.
\end{lemmma}

\begin{proof}
Fix $T\in\GL_n(q)$ such that $v_i^{(1)}=T^{-1}e_i$ for $i=1,\ldots,r_1$, where $e_i$ denotes the $i$-th \enquote{standard basis vector} of $\IF_q^n$ (which has $i$-th entry $1$ and all other entries $0$). Then for $i=1,\ldots,r_1$, the condition $gv_i^{(1)}=w_i^{(1)}$ is equivalent to

\begin{equation}\label{hEq1}
he_i=y_i^{(1)},
\end{equation}

where $h:=TgT^{-1}$ and $y_i^{(1)}:=Tw_i^{(1)}$. Furthermore, for $j=1,\ldots,r_2$, the condition $g^tv_j^{(2)}=w_j^{(2)}$ is equivalent to

\begin{equation}\label{hEq2}
h^tx_j^{(2)}=y_j^{(2)},
\end{equation}

where $x_j^{(2)}:=T^tv_j^{(2)}$ and $y_j^{(2)}:=(T^{-1})^tw_j^{(2)}$. Instead of counting the number of $g\in\GL_n(q)$ satisfying the $r_1+r_2$ many mapping conditions from the assumptions, we count the number of $h\in\GL_n(q)$ satisfying all the equivalently rewritten conditions from Formulas (\ref{hEq1}) and (\ref{hEq2}).

To this end, note that each of the conditions $he_i=y_i^{(1)}$, $i=1,\ldots,r_1$, completely determines one of the first $r_1$ many columns of the matrix $h$.

Note further that, since the $x_j^{(2)}=T^tv_j^{(2)}$, $j=1,\ldots,r_2$, are $\IF_q$-linearly independent, there exist indices $1\leq t_1<t_2<\cdots<t_{r_2}\leq n$ such that for $j=1,\ldots,r_2$, a suitable $\IF_q$-linear combination of $x_1^{(2)},\ldots,x_{r_2}^{(2)}$ is a vector $z_j$ whose $i_j$-th coordinate is $1$ and whose $i_k$-th coordinate for $k\in\{1,\ldots,r_2\}\setminus\{j\}$ is $0$. Hence the conditions $h^tx_j^{(2)}=y_j^{(2)}$, $j=1,\ldots,r_2$, together imply conditions of the form

\begin{equation}\label{hEq3}
h^tz_j=u_j,
\end{equation}

where $u_j$ is a suitable linear combination of $y_1^{(2)},\ldots,y_{r_2}^{(2)}$. However, by the conditions from Formula (\ref{hEq3}), the rows number $t_1,\ldots,t_{r_2}$ of $h$ can be expressed as $\IF_q$-linear combinations of the rows of $h$ whose number is not from the set $\{t_1,\ldots,t_{r_2}\}$.

Combining the two statements about how the conditions affect coefficients from $h$, we see that $h$ is completely determined by the conditions from Formulas (\ref{hEq1}) and (\ref{hEq2}) if we additionally fix the coefficients of $h$ that lie neither in one of the first $r_1$ many columns nor in one of the rows number $t_1,\ldots,t_{r_2}$ of $h$. As there are precisely $n^2-(r_1+r_2)n+r_1r_2$ such coeffcients of $h$, there are at most $q^{n^2-(r_1+r_2)n+r_1r_2}$ many $h\in\GL_n(q)$ that satisfy the conditions from Formulas (\ref{hEq1}) and (\ref{hEq2}), as required.
\end{proof}

\begin{proof}[Proof of Theorem \ref{alternativeTheo}(2) for general linear groups]
Let $G=\GL_n(q)$, $n\geq 72(d+1)^2l^2$, and fix automorphisms $\alpha_1,\ldots,\alpha_l\in\A(G)$. We want to show that the size of the fiber $\Phi$ of $1_G=\id$ under $w_G^{(\alpha_1,\ldots,\alpha_l)}$ is at most $|G|^{d-1/(36(d+1)l^2)}$. As the argument is a modification of the one for the other isometry groups given at the end of the last subsection, we will only indicate at which points the argument needs to be altered here:

\begin{itemize}
\item Instead of $L:=\lfloor n/(9l^2)\rfloor$, we set $L:=\lfloor n/(3(d+1)l^2)\rfloor$ here.
\item As we said at the beginning of this subsection, we cannot write $\alpha_k=\conj(U_k)\circ\aut(\sigma_k)$ anymore in general, but we can write $\alpha_k=\conj(U_k)\circ\aut(\sigma_k)\circ\tau^{a_k}$, where $a_k\in\{0,1\}$.
\item Accordingly, we use Lemma \ref{rewritingLem} for the equivalent reformulation of the mapping conditions on the $g_k$. In those cases where the automorphism $\alpha_k$ occurring in the condition involves $\tau$ (i.e., $a_k=1$), the \enquote{mapping direction} in the equivalent reformulation of the condition is turned around. Hence even with our careful choice of $\overline{z}$, we cannot guarantee anymore that for each $k=1,\ldots,d$, the argument vectors in the various reformulated conditions involving either $g_k$ or $g_k^t$ are linearly independent. However, this is not even necessary, since Lemma \ref{substituteLem}, which we want to apply in order to get an upper bound on the number of possibilities for $\vec{g}$, only requires that each of the two sets of argument vectors in conditions involving $g_k$ and $g_k^t$ separately be linearly independent, which is still the case as long as the system of $b_1+\cdots+b_L$ conditions is not contradictory, by an analogous argument.
\item Hence if we denote, for $k=1,\ldots,d$, the number of rewritten conditions involving $g_k$ by $r_1^{(k)}$ and the number of those conditions involving $g_k^t$ by $r_2^{(k)}$, then an application of Lemma \ref{substituteLem} yields that the number of elements of $X$, subject to the choices of $a_{i,i',j,k}$, $b_i$ and $\overline{z}$, is at most

\[
q^{dn^2-(b_1+\cdots+b_L)n+r_1^{(1)}r_2^{(1)}+\cdots+r_1^{(d)}r_2^{(d)}}\leq q^{dn^2-(b_1+\cdots+b_L)n+dl^2L^2}.
\]

Hence we get the following upper bound on $|X|$ here:

\[
|X|\leq l^Lq^{l^2L^2+dn^2+dl^2L^2}=l^Lq^{(d+1)l^2L^2}(q^{n^2})^d.
\]

Now $|G|=|\GL_n(q)|=(q^n-1)(q^n-q)\cdots(q^n-q^{n-1})\geq (q^{n-1})^n=q^{n(n-1)}$, and so $q^{n^2}\leq |G|^{n/(n-1)}=|G|^{1+1/(n-1)}\leq q^{n^2/(n-1)}\cdot|G|$. Therefore,

\begin{equation}\label{xUpperBoundEq3}
|X|\leq l^Lq^{(d+1)l^2L^2+dn^2/(n-1)}|G|^d.
\end{equation}

The lower bound on $|X|$ is still the same as in Formula (\ref{xLowerBoundEq2}).

\item Note that since we are assuming that $n\geq 72(d+1)^2l^2$, we have

\begin{equation}\label{auxiliaryEq}
\frac{2d+1}{n}-\frac{1}{18(d+1)l^2} \leq -\frac{1}{36(d+1)l^2}.
\end{equation}

Indeed, Formula (\ref{auxiliaryEq}) is equivalent to $n\geq 36(2d+1)(d+1)l^2$, and $36(2d+1)(d+1)l^2\leq 36(2d+2)(d+1)l^2=72(d+1)^2l^2$.

Hence by combining the upper and lower bound on $|X|$, we get the following:

\begin{align*}
|\Phi| &\leq l^Lq^{(d+1)l^2L^2-L(n-1)+dn^2/(n-1)}|G|^d \\
&\leq l^{n/(3(d+1)l^2)}q^{(d+1)l^2n^2/(9(d+1)^2l^4)-n/(3(d+1)l^2)\cdot n/2+2nd}|G|^d \\
&\leq q^n\cdot (q^{n^2})^{(d+1)l^2/(9(d+1)^2l^4)-1/(6(d+1)l^2)+2d/n}|G|^d \\
&= (q^{n^2})^{1/(9(d+1)l^2)-1/(6(d+1)l^2)+(2d+1)/n}|G|^d \\
&= (q^{n^2})^{(2d+1)/n-1/(18(d+1)l^2)}|G|^d \\
&\leq (q^{n^2})^{-1/(36(d+1)l^2)}|G|^d \leq |G|^{d-1/(36(d+1)l^2)},
\end{align*}

where the second-to-last $\leq$ (i.e., the first $\leq$ in the last row) is by Formula (\ref{auxiliaryEq}).
\end{itemize}
\end{proof}

\section{Proof of Theorem \ref{mainTheo}}\label{sec4}

For proving Theorem \ref{mainTheo}, we are supposed to exclude certain nonabelian finite simple groups as composition factors of a finite group $G$ satisfying the condition $\p_w(G)\geq\rho$ for some fixed nonempty reduced word $w$ and $\rho\in\left(0,1\right]$.

Assume that $S$ is a nonabelian composition factor of $G$. By Lemma \ref{mainLem}(2), we know that $\p_w(G)\leq\p_w(N)\cdot\p_w(G/N)$ whenever $N$ is characteristic in $G$. It follows that $\rho\leq\p_w(G)\leq\prod_{i=1}^r{\p_w(F_i)}\leq\min_{i=1,\ldots,r}{\p_w(F_i)}$, where $F_1,\ldots,F_r$ are the \emph{characteristic} composition factors of $G$, i.e., the factors in any principal characteristic series of $G$ (see \cite[p.~65]{Rob96a}), counted with multiplicities. As each $F_i$ is characteristically simple and thus of the form $S_i^{n_i}$ for some finite simple group $S_i$ and $n_i\in\IN^+$ by \cite[3.3.15, p.~87]{Rob96a}, there must exist $i\in\{1,\ldots,r\}$ such that $S_i=S$. Hence we can derive from the assumption that $S$ is a composition factor of $G$ that $\p_w(S^n)\geq\rho$ for some $n\in\IN^+$.

Our next goal on the way to the proof of Theorem \ref{mainTheo} thus is to study $\p_w(T)$, where $T=S^n$ is a finite nonabelian characteristically simple group. In Lemma \ref{variationLem} below, we will show that $\p_w(S^n)\leq\max_{w'}{\p_{w'}(S)}$, where $w'$ runs through a finite set of words associated with $w$, the so-called \enquote{variations of $w$}:

\begin{definition}\label{variationDef}
Let $w=x_1^{\epsilon_1}\cdots x_l^{\epsilon_l}=X_{\iota(1)}^{\epsilon_1}\cdots X_{\iota(l)}^{\epsilon_l}$ be a reduced word of length $l\in\IN$ in the variables $X_1,\ldots,X_d$. For $k=1,\ldots,d$, denote by $a_k$ the number of occurrences of $X_k^{\pm1}$ in $w$ (so that $a_1+\cdots+a_d=l$). A \emph{variation of $w$} is a word of the form $X_{\iota(1),t_1}^{\epsilon_1}\cdots X_{\iota(l),t_l}^{\epsilon_l}$, where $t_i\in\{1,\ldots,a_{\iota(i)}\}$ for $i=1,\ldots,l$.
\end{definition}

Hence a variation of $w$ is a word $w'$ of the same length as $w$ and in variables of the form $X_{k,t}$ with $k\in\{1,\ldots,d\}$ and $t\in\{1,\ldots,a_k\}$ that is obtained from $w$ by adding second indices to each occurrence of $X_k^{\pm1}$, $k=1,\ldots,d$, in $w$ such that each second index is from the \enquote{admissible range}, i.e., lies somewhere between $1$ and the number $a_k$ of occurrences of $X_k^{\pm1}$ in $w$.

\begin{example}\label{variationEx}
Consider the commutator word $w=[X_1,X_2]=X_1X_2X_1^{-1}X_2^{-1}$. The following is a variation of $w$: $X_{1,2}X_{2,1}X_{1,1}^{-1}X_{2,1}^{-1}$. The word $X_{1,3}X_{2,1}X_{1,2}^{-1}X_{2,2}^{-1}$, however, is \emph{not} a variation of $w$, since the second index $3$ added to the first variable $X_1$ does not lie within the admissible range $\{1,2\}$.
\end{example}

\begin{remark}\label{variationRem}
Some simple observations concerning variations:

\begin{enumerate}
\item Each variation of a reduced word of length $l$ is again a reduced word of length $l$.
\item Each reduced word $w$ only has finitely many variations. More precisely, if $w$ is a reduced word in the variables $X_1,\ldots,X_d$, and $X_i^{\pm1}$ occurs precisely $a_i$ times in $w$ for $i=1,\ldots,d$, then the number of variations of $w$ is precisely $\prod_{k=1}^d{a_k^{a_k}}$.
\item Each reduced word $w$ can be obtained from each of its variations $w'$ by substituting $X_k$ for $X_{k,t}$, $t=1,\ldots,a_k$, in $w'$. Hence for each finite group $G$ and each variation $w'$ of $w$, $\pi_{w'}(G)=1$ implies $\pi_w(G)=1$, and $\p_{w'}(G)=1$ implies $\p_w(G)=1$.
\end{enumerate}
\end{remark}

\begin{lemma}\label{variationLem}
Let $w$ be a reduced word of length $l\geq 1$ in the variables $X_1,\ldots,X_d$, $S$ a nonabelian finite simple group and $n\in\IN^+$. Set $\epsilon=\epsilon(S,w):=\max_{w'}{\p_{w'}(S)}\in\left(0,1\right]$, where $w'$ runs through the variations of $w$. Then $\p_w(S^n)\leq\epsilon^{\lceil n/l^2\rceil}\leq\epsilon$.
\end{lemma}

\begin{proof}
Fix automorphisms $\vec{\alpha_1},\ldots,\vec{\alpha_l}$ of $S^n$ and an element $\vec{g}=(g_1,\ldots,g_n)$ of $S^n$. By \cite[3.3.20, p.~90]{Rob96a}, we know that $\Aut(S^n)=\Aut(S)\wr\Sym_n$, and so for $i=1,\ldots,l$, we can write $\vec{\alpha_i}=(\alpha_{i,1}\times\cdots\times\alpha_{i,n})\circ\sigma_i$, where each $\alpha_{i,j}$ is an automorphism $S$ and $\sigma_i$ is a coordinate permutation on $S^n$.

Let $\vec{s_1}=(s_{1,1},\ldots,s_{1,n}),\ldots,\vec{s_d}=(s_{d,1},\ldots,s_{d,n})$, where each $s_{k,j}$ is a variable ranging over $S$, so that each $\vec{s_k}$ can be viewed as a variable element of $S^n$. We want to bound the number of solutions in $(S^n)^d\cong S^{nd}$ of the equation

\begin{equation}\label{vectorEq}
w_{S^n}^{(\vec{\alpha_1},\ldots,\vec{\alpha_l})}(\vec{s_1},\ldots,\vec{s_d})=\vec{g}.
\end{equation}

As usual, let us write $w=x_1^{\epsilon_1}\cdots x_l^{\epsilon_l}=X_{\iota(1)}^{\epsilon_1}\cdots X_{\iota(l)}^{\epsilon_l}$. By computing the LHS in Formula (\ref{vectorEq}) and comparing the entries of the vectors on both sides of the resulting equation, we see that the equation in Formula (\ref{vectorEq}) is equivalent to the conjunction of the following $n$ \enquote{coordinate equations}, for $i=1,\ldots,n$:

\begin{equation}\label{componentEq}
\alpha_{1,i}(s_{\iota(1),\sigma_1^{-1}(i)})^{\epsilon_1}\cdots\alpha_{l,i}(s_{\iota(l),\sigma_l^{-1}(i)})^{\epsilon_l}=g_i.
\end{equation}

The LHS of each of these equations is, up to a suitable renaming of the variables, the evaluation of an automorphic word map associated with a variation of $w$ in variables ranging over $S$. In particular, if $J_i$ denotes the set of those variables $s_{k,j}$ that are mentioned in the $i$-th coordinate equation, then that same equation implies that if we project the solution set $\Phi$ to the equation in Formula (\ref{vectorEq}) onto those coordinates that correspond to variables from $J_i$, the resulting image has size at most $\epsilon|S|^{|J_i|}$.

Our goal is to find $\lceil n/l^2\rceil$ pairwise distinct indices $i_1,\ldots,i_{\lfloor n/l^2\rfloor}\in\{1,\ldots,n\}$ such that the associated coordinate equations are pairwise independent, i.e., such that $J_{i_t}\cap J_{i_u}=\emptyset$ for $t\not=u$. Once we have found these indices, we are done, since it then follows that the projection of $\Phi$ onto those coordinates that correspond to variables from $\bigcup_{t=1}^{\lfloor n/l^2\rfloor}{J_{i_t}}$ has size at most

\[
\prod_{t=1}^{\lceil n/l^2\rceil}{\epsilon|S|^{|J_{i_t}|}}=\epsilon^{\lceil n/l^2\rceil}|S|^{|J_{i_1}|+\cdots+|J_{i_{\lceil n/l^2\rceil}}|},
\]

and thus $\Phi$ itself has size at most

\[
\epsilon^{\lceil n/l^2\rceil}|S|^{|J_{i_1}|+\cdots+|J_{i_{\lceil n/l^2\rceil}}|}\cdot |S|^{n-(|J_{i_1}|+\cdots+|J_{i_{\lceil n/l^2\rceil}}|)}=\epsilon^{\lceil n/l^2\rceil}|S|^n,
\]

as required.

We choose the indices $i_1,\ldots,i_{\lceil n/l^2\rceil}$ iteratively. $i_1$ can be chosen arbitrarily from $\{1,\ldots,n\}$. Denote by $M_1\subseteq\{1,\ldots,n\}$ the set of second indices $j$ in variables $s_{k,j}$ that are mentioned in the $i_1$-th coordinate equation, and note that $|M_1|\leq l$. Note that for any $i_2\in\{1,\ldots,n\}$, independence of the $i_1$-th and $i_2$-th coordinate equation is guaranteed if the sets of second indices that occur in the two equations are disjoint. Now if $M_2$ denotes the set of second indices occurring in the $i_2$-th equation, for $M_2$ to be disjoint with $M_1$, we need that $i_2$ does not lie in the set $\bigcup_{i=1}^l{\sigma_i[M_1]}$, which has size at most $l^2$. Hence as long as $n>l^2$, i.e., $\lceil n/l^2\rceil\geq 2$, we can choose such a second index $i_2$. More generally, if we have already found indices $i_1,\ldots,i_t$ such that the associated coordinate equations are pairwise independent and we want to find another index $i_{t+1}$, it is sufficient to choose $i_{t+1}$ outside of the set $\bigcup_{i=1}^l{\sigma_i[\bigcup_{u=1}^t{M_u}]}$, where $M_u$ denotes the set of second indices occurring in the $i_u$-th equation. This set of \enquote{forbidden} values for $i_{t+1}$ has size at most $tl^2$, and so as long as $n>tl^2$, i.e., $\lceil n/l^2\rceil\geq t+1$, we can choose $i_{t+1}$ as desired. This concludes the proof.
\end{proof}

The proof of Theorem \ref{mainTheo} is now easy:

\begin{proof}[Proof of Theorem \ref{mainTheo}]
For (1): If $S=\Alt_m$ is a composition factor of $G$, then by the observations from the beginning of this subsection, it follows that $\p_w(\Alt_m^n)\geq\rho$ for some $n\in\IN^+$, and thus $\p_{w'}(\Alt_m)\geq\rho$ for some variation $w'$ of $w$. However, $w'$ is a reduced word of length $l$ in at most $l$ distinct variables, and so if $|S|=|\Alt_m|>\max\{\lceil 256l^{16}e^{16M'l-2}\rceil!,\rho^{-16M'}\}$, we get a contradiction, since this implies by Theorem \ref{simpleGroupsTheo}(1) that

\[
\rho\leq\p_{w'}(\Alt_m)\leq|\Alt_m|^{-1/(16M')}<(\rho^{-16M'})^{-1/(16M')}=\rho.
\]

For (2): Assume that $S=X_r(q)$ is a (classical) simple group of Lie type with $r>\max\{72(l+1)^2l^2,\sqrt{72(l+1)l^2\log_2(\rho^{-1})}\}$ and that $S$ is a composition factor of $G$. As before, it follows that $\p_{w'}(S)\geq\rho$ for some variation $w'$ of $w$. In view of our choice of $r$, and using again that $w'$ is a reduced word of length $l$ in at most $l$ distinct variables and that $|X_r(q)|\geq q^{r^2}\geq 2^{r^2}$ (which follows from the known formulas for $|X_r(q)|$, for example from \cite[Table 6, p.~xvi]{CCNPW85a}), we get by Theorem \ref{simpleGroupsTheo}(2) that

\[
\rho\leq\p_{w'}(X_r(q))\leq|X_r(q)|^{-1/(72(l+1)l^2)}\leq 2^{-r^2/(72(l+1)l^2)}<\rho,
\]

a contradiction.
\end{proof}

\begin{proof}[Proof of Theorem \ref{wordMapTheo}]
This follows immediately from Theorem \ref{mainTheo}, as $\Pi_w(G)\leq\P_w(G)$.
\end{proof}

\section{Concluding remarks}\label{sec5}

As mentioned at the beginning of Section \ref{sec3}, the generalization of the third case in Larsen and Shalev's proof (the simple Lie type groups of bounded rank) from the word map setting to automorphic word maps is open. Described very briefly, Larsen and Shalev's approach to the third case is an algebro-geometric one and consists in studying the fibers of word maps in simple Lie type groups as subvarieties of the Lie type groups viewed as linear algebraic groups. One of the problems with extending this approach to automorphic word maps is that because of the existence of field automorphisms on Lie type groups, the degrees of the polynomial equations defining the fiber as a variety are, in contrast to the word map setting, in general not bounded by a constant any more.

Still, hoping that this and other difficulties can be overcome with sufficiently refined ideas, we will spend the rest of this concluding section discussing possible consequences of a successful adaptation of the proof.

The following is a direct generalization of \cite[Theorem 1.1]{LS12a} to automorphic word maps and would most likely result from a suitable adaptation of their proof in its entirety:

\begin{conjecture}\label{larShaConj}
For each nonempty and reduced word $w$ in $d$ distinct variables, there exist constants $N(w),\eta(w)>0$ such that for all nonabelian finite simple groups $S$ with $|S|\geq N(w)$, the inequality $\P_w(S)\leq |S|^{d-\eta(w)}$ holds.
\end{conjecture}

Consider also the following slightly stronger version of Conjecture \ref{larShaConj}:

\begin{conjecture}\label{larShaEffConj}
Like Conjecture \ref{larShaConj}, but with the additional assumption that the constants $N(w)$ and $\eta(w)$ are \emph{effective}, i.e., they can be computed algorithmically from the word $w$ as input.
\end{conjecture}

Our last goal in this paper is to show that Conjecture \ref{larShaEffConj} implies another interesting statement, given as Conjecture \ref{radicalConj} below. Before this, for the readers' convenience, we briefly review some basic facts on the solvable radical and finite groups with trivial solvable radical (for more details, readers are referred to \cite[pp.~88ff. and p.~122]{Rob96a}), and we give some motivation.

Recall that every finite group $G$ has a largest solvable normal subgroup, called the \emph{solvable radical of $G$} and denoted by $\Rad(G)$. The quotient $G/\Rad(G)$ is \emph{semisimple}, i.e., it has no nontrivial solvable normal subgroups at all. It can be shown that the socle $\Soc(H)$ (the subgroup generated by all the minimal nontrivial normal subgroups) of a finite semisimple group $H$ is isomorphic with a \emph{centerless CR-group}, i.e., a direct product of nonabelian finite simple groups, and that $H$ acts faithfully on $\Soc(H)$ via conjugation, so that $H$ is isomorphic with a subgroup of $\Aut(\Soc(H))$ containing $\Inn(\Soc(H))\cong\Soc(H)$. Conversely, if $R$ is a finite centerless CR-group, and $\Inn(R)\leq G\leq\Aut(R)$, then $G$ is semisimple and $\Soc(G)=\Inn(R)\cong R$. Hence the finite semisimple groups are, up to isomorphism, just those finite groups that occur in between the inner and the full automorphism group of a finite centerless CR-group.

The index $[G:\Rad(G)]$ is clearly an upper bound on the product of the orders of all the nonabelian composition factors of $G$ (counted with multiplicities), so that deriving an upper bound on it means establishing some heavy restrictions on the structure of $G$.

It would be nice if we had an algorithmic method to decide in general for a given reduced word $w$ whether a condition of the form $\p_w(G)\geq\rho$ is always strong enough to imply that $[G:\Rad(G)]$ is bounded in terms of $w$ and $\rho$ or not. This is the case if Conjecture \ref{larShaEffConj} holds true.

\begin{conjecture}\label{radicalConj}
There exists an algorithm which, on input a reduced word $w$, achieves the following:

\begin{itemize}
\item It decides whether there exists a function $g_w:\left(0,1\right]\rightarrow\left[1,\infty\right)$ such that for all finite groups $G$ and all $\rho\in\left(0,1\right]$, if $\p_w(G)\geq\rho$, then $[G:\Rad(G)]\leq g_w(\rho)$.
\item In case such a function $g_w$ exists, it also outputs a definition for a possible choice of $g_w$.
\end{itemize}
\end{conjecture}

\begin{proof}[Proof that Conjecture \ref{larShaEffConj} implies Conjecture \ref{radicalConj}]
Write $\Soc(G/\Rad(G))=S_1^{n_1}\times\cdots\times S_r^{n_r}$, where the $S_i$ are pairwise nonisomorphic nonabelian finite simple groups. Note that each $S_i^{n_i}$ is a characteristic composition factor of $G$, and so $\rho\leq\p_w(G)\leq\p_w(S_i^{n_i})\leq\max_{w'}{\p_{w'}(S_i)}$ for $i=1,\ldots,r$, where $w'$ runs through the variations of $w$.

Compute $N_0(w):=\max_{w'}{N(w')}$ and $\eta_0:=\min_{w'}{\eta(w')}$, and note that necessarily $\max_{i=1,\ldots,r}{|S_i|}\leq\max\{N_0(w),\rho^{-1/\eta_0(w)}\}$, as otherwise, if $|S_i|$ is strictly larger than that maximum, it follows that $\rho\leq\max_{w'}{\p_{w'}(S_i)}\leq|S_i|^{-\eta_0(w)}<(\rho^{-1/\eta_0(w)})^{-\eta_0(w)}=\rho$, a contradiction.

Hence we can effectively reduce the list of nonabelian finite simple groups $S$ that could potentially occur as a factor of $\Soc(G/\Rad(G))$ to a finite number of possibilities. There are two cases to consider:

\begin{enumerate}
\item For one of those finitely many nonabelian finite simple groups $S$, we have $\p_w(S)=1$. In other words, there exist automorphisms $\alpha_1,\ldots,\alpha_l$ of $S$ such that $w_S^{(\alpha_1,\ldots,\alpha_l)}$ is constant on $S^d$. Then it is easy to see that $w_{S^n}^{(\alpha_1^{(n)},\ldots,\alpha_l^{(n)})}$ is constant on $(S^n)^d$, and so $\p_w(S^n)=1$ for all $n\in\IN^+$. Hence in that case, $[G:\Rad(G)]$ cannot be bounded under any of the assumptions $\p_w(G)\geq\rho$, $\rho\in\left(0,1\right]$.
\item For each of these finitely many $S$, $\p_w(S)<1$. Then for every variation $w'$ of $w$, $\p_{w'}(S)<1$ as well, by Remark \ref{variationRem}(3). Hence

\[
\epsilon=\epsilon(S,w):=\max_{w'}{\p_{w'}(S)}\leq 1-1/|S|^l.
\]

Therefore, by Lemma \ref{variationLem}, $\p_w(S^n)\geq\rho$ implies

\[
n\leq n_0(w,\rho):=\lfloor l^2\cdot\log(\rho)/\log(1-1/|S|^l)\rfloor.
\]

It follows that $|\Soc(G/\Rad(G))|$ is effectively bounded from above in terms of $w$ and $\rho$, namely by $\prod_{S}{|S|^{n_0(w,\rho)}}$, where $S$ runs through the nonabelian finite simple groups of order at most $\max\{N_0(w),\rho^{-1/\eta_0(w)}\}$. Since $G/\Rad(G)$ embeds into $\Aut(\Soc(G/\Rad(G)))$, its order is thus also effectively bounded in terms of $w$ and $\rho$; more precisely,

\begin{align*}
|G/\Rad(G)| &\leq|\Aut(\prod_S{|S|^{n_0(w,\rho)}})|=|\prod_S{\Aut(S)\wr\Sym_{n_0(w,\rho)}}| \\
&=\prod_S{(|\Aut(S)|^{n_0(w,\rho)}\cdot n_0(w,\rho)!)}.
\end{align*}
\end{enumerate}

We can thus conclude the proof by noting that it can be effectively decided which of the two cases occurs (just go through the effective finite list of groups $S$ and check for each of them, if necessary by brute force, whether $\p_w(S)=1$).
\end{proof}

\end{document}